\documentclass[11pt]{article}
\usepackage[english]{babel}
\usepackage[utf8]{inputenc}
\usepackage{enumitem} 
\usepackage[letterpaper, top=3cm,bottom=3cm,left=3cm,right=3cm,marginparwidth=1.75cm]{geometry}

\usepackage{amsmath,amssymb,amsfonts,amsthm}
\usepackage{algorithmic}
\usepackage{graphicx}
\usepackage{textcomp}
\usepackage{xcolor}
\usepackage{stackrel}
\usepackage{hyperref}
\usepackage{dsfont}
\usepackage[all]{xy}

\usepackage[letterpaper, top=3cm,bottom=3cm,left=3cm,right=3cm,marginparwidth=1.75cm]{geometry}

 \newcommand{\R}{\mathbb{R}}
 
 \newcommand{\N}{\mathbb{N}}
 \newcommand{\Z}{\mathbb{Z}}
 \newcommand{\f}{\frac}

    \newcommand{\E}{\mathbb{E}}
 \newcommand{\PP}{\mathbb{P}}

 \def \bj{\bar{j}}

\def \bk{\bar{k}}

\newtheorem{theorem}{Theorem}[section]

\newtheorem{definition}[theorem]{Definition}
\newtheorem{lemma}[theorem]{Lemma}
\newtheorem{proposition}[theorem]{Proposition}
\newtheorem{remark}[theorem]{Remark}

\title{Regularity of Weighted Tensorized Fractional Brownian Fields and associated function spaces}

\author{C\'eline Esser$^{a}$, Laurent Loosveldt$^{a}$$^{*}$, B\'eatrice Vedel$^{b}$ \\
        \small $^{a}$Univerist\'e de Li\`ege \\
        \small $^{b}$Univerist\'e de Bretagne-Sud \\\\
        \small $^{*}$Corresponding author \tt{l.loosveldt@uliege.be} \\
}
\date{\today}

\begin{document}
\maketitle
\begin{abstract} 
\noindent 
We investigate a new class of self-similar fractional Brownian fields, called Weighted Tensorized Fractional Brownian Fields (WTFBS). These fields, introduced in the companion paper \cite{ELLV}, generalize the well-known fractional Brownian sheet (FBs) by relaxing its tensor-product structure, resulting in new self-similar Gaussian fields with stationary rectangular increments that differ from the FBs. We analyze the local regularity properties of these fields and introduce a new concept of regularity through the definition of Weighted Tensorized Besov Spaces. These spaces combine aspects of mixed dominating smoothness spaces and hyperbolic Besov spaces, which are similar in structure to classical Besov spaces. We provide a detailed characterization of these spaces using Littlewood-Paley theory and hyperbolic wavelet analysis.
\end{abstract}

\noindent{\bf{keywords :}} Brownian fields, Brownian sheet, rectangular increments, hyperbolic wavelets, Besov spaces\\

\section{Introduction}

Modeling some phenomena, such as the movement of particules observed by Brown in 1827, with the help of random functions, has a long history. 
The first and most well-known model is the Brownian motion, which has been extensively studied. Notably, as early as 1937, Paul Levy established the H\"older regularity properties of its sample paths. He achieved this for by expanding the Brownian motion in  the Schauder-Faber system \cite{Levy37}, whose bases functions are the primitive of the Haar wavelets \cite{Fab09,Haar10} and hence enter the class of biorthogonal vaguelet bases. 

Using also the expansion of the Brownian motion in the Faber-Schauder system, Ciesielski proved that almost surely the sample paths belong to the Besov spaces $B^{1/2}_{p,\infty}$ for $1 \le p<+\infty$, see \cite{Cie92}. 
 A modern and self-contained version of these results is presented in \cite{Kempka}, where regularity is explored within the framework of Orlicz-Besov spaces, as well as for newer, smaller spaces.

The Brownian motion $B$ can be defined as the unique Gaussian process with stationary increments which satisfy the self-similarity property 
$$
\forall a>0, \quad
B_{at} \stackrel{(d)}{=} a^{1/2}B_t.
$$
This property has been naturally extended by Kolmogorov \cite{MR0003441} to define fractional brownian motions $B^H$, which have been systematically studied by Mandelbrot and Van Ness \cite{MN68}. They form a family of Gaussian processes with again stationary increments, depending of a parameter $H \in (0,1)$ called Hurst exponent which the generalized notion of self-similarity
$$
    \forall a>0, \quad B^H_{at} \stackrel{(d)}{=} a^{H}B_t.
$$
In \cite{MST99}, wavelet-type expansions of the fractional Brownian motions are given. The main difficulty when $H \neq 1/2$ is that we have either to deal with fractional primitive of wavelets which are no more compactly supported and might create infrared divergence or to analyze or expand the field in a wavelet basis, but with correlated coefficients. Nevertheless, as for the Brownian motion, the properties of regularity of the sample paths are now well understood, in particular, a.s. the sample paths of $B^H$ belong to the H\"older space ${\mathcal{C}}^{H-\varepsilon}$ for every $\varepsilon>0$ (on any compact set), but does not belong to ${\mathcal{C}}^{H}$. Using precise estimates of these expansions, the recent work in \cite{EL22} has highlighted the existence of both rapid and slow points in fractional Brownian motions.

In higher dimensions, particularly in two dimensions, stochastic fields are widely used to model textures in various application contexts such as medical image analysis, texture synthesis, and more. The generalization of Brownian processes has evolved in different directions, depending on the required properties of the process to match the characteristics of the modeled textures.
In particular, two main models have emerged: fractional Brownian fields, which are isotropic, and fractional Brownian sheets, studied notably by A. Kamont \cite{Kam96}, which introduce intrinsic anisotropy through different regularities along distinct directions. This anisotropy proves useful in modeling certain structures, such as medical tissues (e.g., bones) or hydrological phenomena.\\

$\bullet$ {\bf The fractional brownian fields} $Y^H$, $H \in (0,1)$ also called Levy fractional brownian motion \cite{Samorodnitsky}. They are Gaussian, self-similar with stationary increments and isotropic, meaning that the field is invariant in law under rotations. As a centered Gaussian field, it is characterized by its covariance operator 
$$
\E[ Y^H_x Y^H_y] =\frac{1}{2}\left(\Vert x \Vert^{2H}+ \Vert y \Vert^{2H} - \Vert x -y \Vert^{2H}\right).
$$
It is also characterized by its harmonizable representation, that is
\begin{equation}
Y^H_{\bf{x}}= \int_{\mathbb{R}^N} \frac{e^{i\langle \bf{x} , \boldsymbol{\xi}\rangle} -1}{ \|\boldsymbol{\xi} \|^{H+\frac{N}{2}}}  d\hat{\bf{W}}({\boldsymbol{\xi}}) 
\end{equation}
where $\langle \cdot, \cdot \rangle$ denotes the standard scalar product in $\mathbb{R}^N$ and where $\hat{\bf{W}}$ can be understood as the ``Fourier'' transform of the N-dimensional Brownian measure $\bf{W}$ on an underlying probability space $(\Omega,\mathcal{F},\PP)$, see \cite{MR3839281} for a precise definition. 
On any non trivial compact $K$, the sample paths of the field a.s. do not belong to ${\mathcal{C}}^H$ but belong to ${\mathcal{C}}^{H-\varepsilon}$ for any $\varepsilon>0$. Once again, regularity and irregularity results can be otained by performing a classical N-dimensional wavelet analysis of the fields, i.e. with an orthonormal basis of the form $\{\varphi(\cdot - \bold{k}), \,  \bold{k} \in \R^N\} \cup \{\psi(2^j \cdot - \bold{k}), j \ge 0, \bold{k} \in \R^N\}$, where $\varphi$ is a scaling function and $\psi$ is the mother wavelet, with good properties of localization, regularity and oscillations (see \cite{Daubechies:88,Daubechies:92,LM86,Mallat:99} for constructions of such bases and introduction to wavelet theory). In particular, the Hurst exponent of the field can be determined by 
$$
H = \sup \{\alpha>0 :  \, Y^H \in {\mathcal{C}}^{\alpha}([0,1])  \},
$$
and estimated with log-log regression on wavelet coefficients (see \cite{992817}).

These fields have several extensions, introducing anisotropy in the model \cite{BE03} or local fluctuations in the Hurst exponent \cite{MR3839281}, implying that the regularity of the field varies from point to point. An other important extension is the notion of Operator Scaling Gaussian Fields (OSGF) introduced  in \cite{BMS07,BMBS06}. They satisfy a matricial self-similarity condition, which is given by
$$
\forall a>0, \quad Z_{a^E {\bf{x}}} \stackrel{(d)}{=} a^H Z_{\bf{x}}
$$
for some $H>0$, where $E$ is a $N \times N$ matrix with eigenvalues having positive real parts, and where
$\displaystyle{
a^E = \exp(E \ln(a)) = \sum_{k \ge 0} \frac{\ln ^k (a) E^k}{k!}.}
$
In this context, the natural notion of regularity is no longer the classical one, and it becomes necessary to consider anisotropic functional spaces \cite{BL09, CV10}. These spaces have been extensively studied in \cite{Tri06} and possess biorthogonal wavelet bases, referred to as anisotropic wavelet bases \cite{Tri04}. This framework provides strategies for numerical estimation of model parameters \cite{BR10, RCVJA13}.\\

$\bullet$ {\bf The fractional brownian sheets (FBs)} $S^{\bold{H}}$ \cite{Kam96,MR1906407}. For a given vector ${\bf{H}}= (H_1,\dots, H_N) \in (0,1)^N $, the fBs of Hurst index ${\bf{H}}$ is a real-valued centered Gaussian random field $S^{\bf{H}}$ with covariance function given by
$$
\mathbb{E}[ S^{\bf{H}}_{\bf{x}} S^{\bf{H}}_{\bf{y}}] \!\!= \!\!\!\prod_{m=1}^N \frac{1}{2} \left( \vert x_m \vert^{2H_m} + \vert y_m \vert^{2H_m} - \vert x_m -y_m \vert^{2H_m}\right).
$$
It can also be characterized by its harmonizable representation
\begin{equation}
S^{\bf{H}}_{\bf{x}} = \int_{\mathbb{R}^N} \prod_{m=1}^N \frac{ e^{i  x_m \xi_m }-1}{\vert \xi_m\vert^{H_m +\frac{1}{2}} } d\hat{\bf{W}}({\boldsymbol{\xi}}).
\end{equation}
Setting $H_m=\frac{1}{2}$ for each $m \in \{1, \dots, N\}$ yields the standard Brownian sheet. 

Classical spaces and even anisotropic ones are not well-suited to study the regularity of these fields. Indeed, due to the tensor-product structure in the covariance -- and similarly in the kernel of the harmonizable representation --  it is natural to study the regularity of the field in the scale of spaces of mixed dominating smoothness \cite{Vybiral} and to characterize it with the help of hyperbolic wavelet \cite{Kam96}.

Note that the fractional Brownian sheet $S^{\bold{H}}$ satisfies the self-similarity property 
$$
\forall a >0, \quad S^{\bold{H}}_a{\bold{x}} \stackrel{(d)}{=}  a^{\bold{H}} S_{\bold{x}}^H
$$
and its rectangular increments -- as defined in Section \ref{sec:kolmogorov} -- are stationary.
Whereas fractional Brownian fields are the unique Gaussian self-similar fields with stationary increments, Makogin and Mishura have exhibited in \cite{MR3339311, MR4026763} self-similar Gaussian fields with stationary rectangular increments which are distinct from the FBs. Here, the notion of self-similarity required for the field $Z$ is
$$
 \forall a_1, a_2>0, \quad Z_{(a_1x_1,a_2x_2)}\stackrel{(d)}{=} a_1^{H_1} a_2^{H_2} Z(x_1,x_2).
$$

In this paper, we study a new class of self-similar fractional Brownian fields, called {\bf{Weighted Tensorized Fractional Brownian Fields}} (WTFBFs). These fields were introduced in a short companion conference paper \cite{ELLV}, where numerical simulations were  provided. Notably, these simulations highlighted how the parameter $\alpha$ contributes to relaxing the tensor-like structure of the field.
Our motivation for this study is multifaceted :
\begin{enumerate}[label=\alph*)]
    \item These fields offer new examples of self-similar fields with stationary rectangular increments, complementing the examples presented in \cite{MR3339311, MR4026763}, when the notion of self-similarity is relaxed to:
    $$
    \forall a>0 , \quad Z_{a\bold{x}} \stackrel{(d)}{=}  a^H Z_{\bold{x}}.
    $$
    These fields have a relatively simple spectral representation (cf. below), and it allows us to obtain both statistical and regularities properties. 
    \item The fractional Brownian Sheet has been introduced to model anisotropic textures such as bones in the diagnosis of osteoporosis \cite{Kam96}. However, due to its  strong tensor-product structure, it is  not always an appropriate model for real-world data. Nonetheless, in certain contexts, particularly in modeling  reticulated textures -- such as textiles, biological structures and urban networks -- a controlled ``tensor-product''-like property can be useful. Our new class of fields offers flexibility, ranging from  strongly tensorized fields to nearly isotropic ones, depending on a parameter $\alpha$. The endpoint $\alpha=0$ corresponds to the classical FBs when $\alpha=1$ yields a field that closely resembles to the FBf, particularly in term of regularity.
    \item The relaxation of the strict tensor-product structure renders both classical and mixed dominating smoothness notions inadequate for studying these fields. Therefore, we introduce a new notion of regularity and define {\bf Weighted Tensorized Besov Spaces}. These spaces are a hybrid between spaces of mixed dominating smoothness (at $\alpha=0$) and hyperbolic Besov spaces (at $\alpha=1$), the latter being quite similar to classical Besov spaces.

    This question of modulating or weighting the tensor-product effect echoes works on PDE's where the physically relevant solutions of a electronic Schr\"odinger equation naturally has this kind of hybrid regularity \cite{Yse1,Yse2}. This hybrid smoothness can be then used to reduce numerical efforts to compute solutions, and it leads to ongoing researches on non-linear approximation on these spaces in \cite{BHW,Har24}, where the spaces are only introduced via conditions on the hyperbolic wavelet coefficients of their functions. 
\end{enumerate}
    As a matter of fact, the introduction of hyperbolic spaces with characterization in terms of hyperbolic wavelets $X^{s}_{p,q}$ ($X=B$ for Besov spaces or $F$ for Triebel-Lizorkhin spaces) in \cite{ACJRV, SUV}  has proven useful. These spaces are equivalent to classical spaces if and only if $p=q=2$, but they remain very close in other cases, differing only by a logarithmic correction. This approach provides a unified tool for analyzing isotropic, anisotropic, and tensor-product-like structures. 
    In this paper, we aim to present various equivalent definitions of these spaces (through finite differences, Littlewood-Paley analysis, and wavelet characterizations), explore properties of embeddings, and showcase the typical Gaussian random processes related to this notion of smoothness. Let us also mention related works where different extensions of the notions of smoothness have been studied such as directional regularities and rectangular pointwise regularity \cite{BABS,BSBB}. 

The paper is organised as follows : After quickly giving the definitions and the first properties of the field below, Section 2 is devoted to a variant of Kolmogorov's continuity Theorem and provide us with the regularity of the fields. Irregularity properties are obtained in Section 3 and Section 4 is devoted to the study of the associated Besov spaces.

\begin{definition}[\bf Weighted Tensorized fractional Brownian fields]
For $\alpha \in [0,1]$ and $H \in (0,1)$, we set
\begin{equation}\label{eq:H+H-}
H_\alpha^+ := (1+\alpha)H \quad \text{ and } \quad H_\alpha^- := (1-\alpha)H    
\end{equation} 
and we define the Gaussian field $\{X^{\alpha,H}_{(x_1,x_2)}\}_{(x_1,x_2) \in \mathbb{R}^2}$ by 
\begin{equation}\label{eqn:defchamp}
X^{\alpha,H}_{(x_1,x_2)} :=  \int_{\mathbb{R}^2}  \frac{(e^{i x_1 \xi_1}-1) (e^{i x_2 \xi_2}-1)}{\phi_{\alpha,H}(\xi_1,\xi_2)}  d\hat{\bf{W}}({\boldsymbol{\xi}}) 
\end{equation}
where the function
\[\phi_{\alpha,H}(\xi_1,\xi_2)=\min(\vert \xi_1 \vert, \vert \xi_2 \vert)^{H_\alpha^- +\frac{1}{2}}\max (\vert \xi_1 \vert, \vert \xi_2 \vert)^{H_\alpha^++\frac{1}{2} }\]
denotes the square root of the inverse of the spectral density of the field.
\end{definition}
In the sequel, we also use the notation
\[ \mathcal{K}^{\alpha,H}_{(x_1,x_2)}(\xi_1,\xi_2) := \frac{(e^{i x_1 \xi_1}-1) (e^{i x_2 \xi_2}-1)}{\phi_{\alpha,H}(\xi_1,\xi_2)} \]
for the kernel in the stochastic integral \eqref{eqn:defchamp}. Note that the field \eqref{eqn:defchamp} is well-defined according to the following Lemma. 

\begin{lemma}\label{lem:biendefini}
The kernel $\mathcal{K}^{\alpha,H}$ is in $L^2(\R^2)$.
\end{lemma}
\begin{proof}
The kernel being symmetric in $\xi_1$ and $\xi_2$, we can restrict the domain of integration to the half plane $\{|\xi_1| \le | \xi_2| \}$. One writes, on a neighborhood of $(0,0)$,
\begin{eqnarray*}
\left\vert \frac{(e^{ix_1 \xi_1}-1)(e^{i x_2 \xi_2} -1)}{\vert \xi_1 \vert ^{(1-\alpha)H + 1/2}\vert \xi_2 \vert^{(1+\alpha)H+1/2}} \right\vert^2 & \le & \frac{C \vert x_1  \xi_1 \vert^2 \vert x_2 \xi_2 \vert^2 }{\vert \xi_1 \vert^{2(1-\alpha)H+1} \vert \xi_2 \vert^{2(1+\alpha)H +1}} \\
& \le & \frac{C \vert x_1  \xi_1 \vert^2 \vert x_2 \xi_2 \vert^2 }{\vert \xi_1 \vert^{2H+1} \vert \xi_2 \vert^{2H +1} }
\end{eqnarray*}
using $\vert \xi _1 \vert \le \vert \xi_2 \vert$ for this last inequality. This gives the integrability in $(0,0)$. A similar argument gives the integrability when $\Vert \xi \Vert \to + \infty.$
\end{proof}

Let us end this introduction by mentioning the first basic properties of the field $X^{\alpha, H}$, proved in  \cite{ELLV}.

\begin{proposition}
    For all $\alpha \in [0,1]$ and all $H \in (0,1)$, the field  $X^{\alpha, H}$ is real, self-similar with exponent $2H$ and has stationary rectangular increments. 
\end{proposition}

\section{Regularity of the sample paths: A variant of Kolmogorov's continuity Theorem}\label{sec:kolmogorov}

A usual strategy to get a first glimpse on the regularity of a stochastic process from very basic probabilistic quantity is to use the Kolmogorov's continuity Theorem. In its most basic form, it tells that if a  field $\{X_{\mathbf{x}}\}_{\mathbf{x} \in \R^d}$ is such that  there exist some constants $\beta, \eta, c, R>0$ for which, for all $\mathbf{x},\mathbf{y}\in \R^2$ with $\|\mathbf{x}-\mathbf{y}\| < R$, we have
\begin{equation}\label{kolmo:intro}
     \E[| X_{\mathbf{x}}-X_{\mathbf{y}}|^\eta] \leq c \left(\|\mathbf{x}-\mathbf{y}\|\right)^{1+\beta}
\end{equation}
  then there exists a version $\{\widetilde{X}_{\mathbf{x}}\}_{\mathbf{x} \in \R^d}$ of $\{X_{\mathbf{x}}\}_{\mathbf{x} \in \R^d}$ which is locally Hölder of order $\gamma$, for all $\gamma \in [0,\frac{\beta}{\eta})$. It means that, for any such $\gamma$ and any bounded set $K$, for all $\omega \in \Omega$, there exists a finite constant $C(\omega)>0$ such that for all $\mathbf{x},\mathbf{y} \in K$,
  \[ |\widetilde{X}_{\mathbf{x}}-\widetilde{X}_{\mathbf{y}} | \leq C(\omega) \left(\|\mathbf{x}-\mathbf{y}\|\right)^{\gamma}.\]
Numerous generalizations of Kolmogorov's continuity theorem (also known as the Kolmogorov-Chentsov theorem) can be found in the literature. See, for instance, \cite{MR4621071} for an extension in the very general context of  stochastic processes defined on a metric space and taking values in another metric space. These kinds of results are particularly well-adapted for processes with stationary increments, as the law of $X_{\mathbf{x}}-X_{\mathbf{y}}$ used in \eqref{kolmo:intro} (or generally $d(X_{\mathbf{x}},X_{\mathbf{y}})$ for processes on a metric space with distance $d$) does not depend on $\|\mathbf{x}-\mathbf{y}\|$. 

  Nevertheless, as already noted in \cite{MR1906407,MR2274895,MR3339311,MR4026763,MR3498033}, for real-valued stochastic field with a ``tensorized structure'', the stationarity of increments is rarely met and it is often preferable to work with so-called rectangular increments. Given a stochastic field $\{X_{\mathbf{x}}\}_{\mathbf{x} \in \R^d}$ and $\mathbf{h} \in \R^d$, the rectangular increment of $X$ at $\mathbf{x} \in \R^d$ with step $\mathbf{h}$ is given by
\[\Delta  X_{\mathbf{h};\mathbf{x}}:= \sum_{ (k_1,\dots,k_d) \in \{0,1\}^d } (-1)^{d-(k_1+\dots+k_d)} X(x_1 + k_1 h_1,\dots,x_d + k_d h_d).\]
In this paper, we work with $d=2$ and therefore we  have
\[\Delta  X_{(h_1,h_2);(x_1,x_2)} := X(x_1+h_1,x_2+h_2)-X(x_1+h_1,x_2)-X(x_1,x_2+h_2)+X(x_1,x_2).\]
Starting from the observation that rectangular increments are an appropriate quantity to study stochastic field with tensorized structure, the authors in \cite{MR1696137,MR1728004} propose the following variant of Kolmogorov's continuity theorem. Let $\{X_{(x_1,x_2)}\}_{(x_1,x_2) \in \R^2}$ be a two-dimensional fied such that  there exist  constants $\beta_1, \beta_2, \eta, c, R>0$ for which, for all $(x_1,x_2)\in \R^2$ and $(h_1,h_2) \in \R^2$ with $|h_1 | < R$ and $|h_2| < R$, we have
$$
     \E[| \Delta  X_{(h_1,h_2);(x_1,x_2)}|^\eta] \leq c |h_1|^{1+\beta_1} |h_2|^{1+\beta_2}.
$$
Then, there exists a version $\{\widetilde{X}_{(x_1,x_2)}\}_{(x_1,x_2) \in \R^2}$ of  $\{X_{(x_1,x_2)}\}_{(x_1,x_2) \in \R^2}$ for which, for all $\gamma_1 \in [0,\frac{\beta_1}{\eta})$ and $\gamma_2 \in [0,\frac{\beta_2}{\eta})$, and for every bounded intervals $I,J$ of $\mathbb{R}$, for all $\omega \in \Omega$, there exists a finite constant $C(\omega)>0$ such that
    \[|\Delta  \widetilde{X}_{(h_1,h_2);(x_1,x_2)}| \leq C |h_1|^{\gamma_1} |h_2|^{\gamma_2} .\]
for every $x_1 \in I$, $x_2 \in J$ and $h_1, h_2 \in \mathbb{R}$ with $x_1+h_1 \in I$ and $x_2 + h_2 \in J$.  This result could be applied to study the regularity of the WTFBFs. However, the weights within the tensorized structure of this field allow to refine the bounds on the moments, as stated in the following proposition.

\begin{proposition}\cite{ELLV}\label{prop:const_increm}
For all {$\alpha \in [0,1]$ and $H \in (0,1)$}, there is a constant $c_1>0$ such that the rectangular increments of $\{X^{\alpha,H}_{(x_1,x_2)}\}_{(x_1,x_2) \in \mathbb{R}^2}$ satisfy 
$$ \mathbb{E}\big[ |\Delta  X^{\alpha,H}_{(h_1,h_2);(x_1,x_2)}|^2 \big]
  \,\,  \leq   c_1 \left(\max\{|h_1|,|h_2| \}^{1-\alpha}\min\{|h_1|,|h_2| \}^{1+\alpha}\right)^{2H}
$$
for all $(x_1,x_2),(h_1,h_2) \in \mathbb{R}^2$. 
\end{proposition}

Our strategy will thus be to exploit this bound on the second moments of rectangular increments of the WTFBFs to establish the regularity of their sample paths. 

\begin{proposition}\label{prop:regularity}
For all $\alpha \in [0,1]$ and $H \in (0,1)$, there exists a version of $X^{\alpha,H}$, which we still denote as $X^{\alpha,H}$, such that for all $\varepsilon>0$ and for every bounded intervals $I,J$ of $\mathbb{R}$, for all $\omega \in \Omega$, there exists a finite constant $C(\omega)>0$ such that
    \[
        |\Delta  X^{\alpha,H}_{(h_1,h_2);(x_1,x_2)}| \leq C(\omega)\left(\max\{|h_1|,|h_2| \}^{1-\alpha}\min\{|h_1|,|h_2| \}^{1+\alpha}\right)^{H-\varepsilon}.
    \]
for every $x_1 \in I$, $x_2 \in J$ and $h_1, h_2 \in \mathbb{R}$ with $x_1+h_1 \in I$ and $x_2 + h_2 \in J$.
\end{proposition}


This result follows directly from a variant of Kolmogorov’s continuity theorem, which we state now.


\begin{theorem}\label{thm:kolmo}
   Let $\alpha \in [0,1]$ be fixed. Assume that $\{X_{(x_1,x_2)}\}_{(x_1,x_2) \in \R^2}$ is a stochastic field for which there exist some constants $\beta, \eta, c, R>0$ such that 
   \begin{equation}\label{eqn:theoremkolmohyp}  \E[|\Delta  X_{(h_1,h_2);(x_1,x_2)}|^\eta] \leq c \left(\max\{|h_1|,|h_2| \}^{1-\alpha}\min\{|h_1|,|h_2| \}^{1+\alpha}\right)^{1+\beta}
    \end{equation}
    for all $(x_1,x_2)\in \R^2$ and $(h_1,h_2) \in \R^2$ with $|h_1| < R$ and $|h_2|<R$. Then, there exists a version $\{\widetilde{X}_{(x_1,x_2)}\}_{(x_1,x_2) \in \R^2}$ of $\{X_{(x_1,x_2)}\}_{(x_1,x_2) \in \R^2}$ such that, for every $\gamma \in [0,\frac{\beta}{\eta})$ and for every bounded intervals $I,J$ of $\mathbb{R}$, for all $\omega \in \Omega$, there exists a finite constant $C(\omega)>0$ such that
    \begin{equation}\label{eqn:theoremkolmo}
        |\Delta  \widetilde{X}_{(h_1,h_2);(x_1,x_2)}| \leq C(\omega)\left(\max\{|h_1|,|h_2| \}^{1-\alpha}\min\{|h_1|,|h_2| \}^{1+\alpha}\right)^\gamma
    \end{equation}
for every $x_1 \in I$, $x_2 \in J$ and $h_1, h_2 \in \mathbb{R}$ with $x_1+h_1 \in I$ and $x_2 + h_2 \in J$.
\end{theorem}

\begin{proof}
    Let us first remark that we can assume, without loss of generality, that the field $\{X_{(x_1,x_2)}\}_{(x_1,x_2) \in \R^2}$ vanishes on the horizontal and vertical axes, i.e., for all $x \in \R$, $X_{(x,0)}=X_{(0,x)}=0$. Indeed, if it is not the case, it suffices to consider the field $\{Y_{(x_1,x_2)}\}_{(x_1,x_2) \in \R^2}$ defined, for all $(x_1,x_2) \in \R^2$, by
    \[ Y_{(x_1,x_2)}=X_{(x_1,x_2)}-X_{(x_1,0)}-X_{(0,x_2)}+X_{(0,0)}\]
    and to note that, for every such  $(x_1,x_2)$ and $(h_1,h_2) \in \R^2$, we have
    \[ \Delta  X_{(h_1,h_2);(x_1,x_2)}= \Delta  Y_{(h_1,h_2);(x_1,x_2)}.\]
Furthermore, since $\R^2$ can be written as a countable union  of rectangle, we can restrict our attention to the subfields  $\{{X}_{(x_1,x_2)}\}_{(x_1,x_2) \in I^2}$, where $I=[a,b] \times [c,d]$. For simplicity, we further restrict our attention to the field defined on $R=[0,1]^2 $. We use the strategy employed, for instance, in \cite{MR3497465} to establish the standard Kolomogrov's continuity Theorem. Namely, we consider the sets
    \[ D := \left\{ \frac{k}{2^j} \, : \, j \in \N, k \in \{0,\dots,2^{j}-1 \}\right\}\]
    of dyadic numbers of $[0,1]$ and prove the announced regularity properties on $D^2$. The version $\{\widetilde{X}_{(x_1,x_2)}\}_{(x_1,x_2) \in [0,1]^2}$ is then constructed by exploiting the density of $D$ in $[0,1]$ and some convergence arguments. Although the structure of the proof is standard, we deal with specific arguments required by the ``rectangular properties'' of the field of interest. 

  Let us fix $\gamma \in [0,\frac{\beta}{\eta})$ for the moment. For every $(j_1,j_2) \in \N^2$, the probability  that there exists $ (k_1,k_2) \in \{0,\dots,2^{j_1}-1\}\times \{0,\dots,2^{j_2}-1\} $ and $a_1,a_2 \in \{-1,1\}$ such that 
  $$|\Delta  X_{(\frac{a_1}{2^{j_1}},\frac{a_2}{2^{j_2}});(\frac{k_1}{2^{j_1}},\frac{k_2}{2^{j_2}})}| \geq 2^{-\gamma(\max\{j_1,j_2\}(1+\alpha)+\min\{j_1,j_2\}(1-\alpha))}
  $$
  is bounded above by 
\begin{align*}
 \sum_{a_1,a_2 \in \{-1,1\}} & \sum_{k_1=0}^{2^{j_1}-1} \sum_{k_2=0}^{2^{j_2}-1} \PP\left(|\Delta  X_{(\frac{a_1}{2^{j_1}},\frac{a_2}{2^{j_2}});(\frac{k_1}{2^{j_1}},\frac{k_2}{2^{j_2}})}| \geq 2^{-\gamma(\max\{j_1,j_2\}(1+\alpha)+\min\{j_1,j_2\}(1-\alpha))}\right) \\ 
         & \leq 4 c \sum_{k_1=0}^{2^{j_1}-1} \sum_{k_2=0}^{2^{j_2}-1} 2^{(\gamma \eta- \beta -1)(\max\{j_1,j_2\}(1+\alpha)+\min\{j_1,j_2\}(1-\alpha))} \\
        & \leq 4 c 2^{(j_1+j_2)(\gamma \eta - \beta)}
 \end{align*}
    using Markov's inequality, where the last inequality follows from
    $$
    2^{- \max\{j_1,j_2\}(1+\alpha)-\min\{j_1,j_2\}(1-\alpha)} \leq 2^{-(j_1+j_2)}.
    $$
 As $\gamma \eta-\beta < 0$, 
 Borel-Cantelli's Lemma entails the existence of an event $\Omega_\gamma$ of probability $1$, such that, for all $\omega \in \Omega_\gamma$ there is $C_\gamma(\omega)>0$ such that, for every $(j_1,j_2) \in \N^2$ and every $ (k_1,k_2) \in \{0,\dots,2^{j_1}-1\}\times \{0,\dots,2^{j_2}-1\} $ and every $a_1,a_2 \in \{-1,1\}$
\begin{equation}\label{eqn:inesurdyadique}
    |\Delta  X_{(\frac{a_1}{2^{j_1}},\frac{a_2}{2^{j_2}});(\frac{k_1}{2^{j_1}},\frac{k_2}{2^{j_2}})}(\omega)| \leq C_\gamma(\omega) 2^{-\gamma(\max\{j_1,j_2\}(1+\alpha)+\min\{j_1,j_2\}(1-\alpha))}.
\end{equation}
Now, let us consider $t_1,t_2,s_1,s_2 \in D$; we will bound from above
\begin{equation} \label{eqn:kolmo:tobound}
    |X_{(t_1,t_2) }-X_{(t_1,s_2) }-X_{(s_1,t_2) }+X_{(s_1,s_2) }|
\end{equation}
on $\Omega_\gamma$.
Let $p_1,p_2 \in \N$ be such that
\begin{equation}
    2^{-(p_1+1)} \leq |t_1-s_1| < 2^{-p_1} \text{ and } 2^{-(p_2+1)} \leq |t_2-s_2| < 2^{-p_2}. 
\end{equation}
It means that we can write
\begin{align*}
    t_1 &= \frac{k_1+\varepsilon_0}{2^{p_1}} + \frac{\varepsilon_1}{2^{p_1 +1}}+\dots +\frac{\varepsilon_{n_1}}{2^{p_1 +n_1}} \\
    s_1 &= \frac{k_1+\varepsilon_0'}{2^{p_1}} + \frac{\varepsilon_1'}{2^{p_1 +1}}+\dots +\frac{\varepsilon_{n_1}'}{2^{p_1 +n_1}} \\
    t_2 &= \frac{k_2+\delta_0}{2^{p_2}} + \frac{\delta_1}{2^{p_2 +1}}+\dots +\frac{\delta_{n_2}}{2^{p_2 +n_2}} \\
    s_2 &= \frac{k_2+\delta_0'}{2^{p_2}} + \frac{\delta_1'}{2^{p_2+1}}+\dots +\frac{\delta_{n_2}'}{2^{p_2 +n_2}} 
\end{align*}
for some $n_1,n_2 \in \N$ and $\varepsilon_0,\varepsilon_1,\dots,\varepsilon_{n_1}; \varepsilon_0',\varepsilon_1',\dots,\varepsilon_{n_1}'; \delta_0,\delta_1,\dots,\delta_{n_2}; \delta_0',\delta_1',\dots,\delta_{n_2}' \in \{0,1\}$ such that $\varepsilon_0=0$ if $t_1 \leq s_1$, $\varepsilon_0'=0$ if $s_1 < t_1$, $\delta_0=0$ if $t_2 \leq s_2$ and $\delta_0'=0$ if $s_2 \leq t_2$. In this case, for all $0 \leq j_1 \leq n_1$ and $0 \leq j_2 \leq n_2$, we write
\begin{align*}
    t_1^{(j_1)} &= \frac{k_1+\varepsilon_0}{2^{p_1}} + \frac{\varepsilon_1}{2^{p_1 +1}}+\dots +\frac{\varepsilon_{j_1}}{2^{p_1 +j_1}} \\
    s_1^{(j_1)} &= \frac{k_1+\varepsilon_0'}{2^{p_1}} + \frac{\varepsilon_1'}{2^{p_1 +1}}+\dots +\frac{\varepsilon_{j_1}'}{2^{p_1 +j_1}} \\
    t_2^{(j_2)} &= \frac{k_2+\delta_0}{2^{p_2}} + \frac{\delta_1}{2^{p_2 +1}}+\dots +\frac{\delta_{j_2}}{2^{p_2 +j_2}} \\
    s_2^{(j_2)} &= \frac{k_2+\delta_0'}{2^{p_2}} + \frac{\delta_1'}{2^{p_2+1}}+\dots +\frac{\delta_{j_2}'}{2^{p_2 +j_2}} 
\end{align*}
and get that \eqref{eqn:kolmo:tobound} can be bounded by 
\begin{align}    
& |X_{(t_1,t_2) }-X_{(t_1,s_2) }-X_{(s_1,t_2) }+X_{(s_1,s_2) }| \nonumber \\ 
     & \leq  |X_{(t_1^{(0)},t_2^{(0)}) }-X_{(t_1^{(0)},s_2^{(0)}) }-X_{(s_1^{(0)},t_2^{(0)}) }+X_{(s_1^{(0)},s_2^{(0)}) }|\label{eqn:3-5}\\
    & + \sum_{j_1=0}^{n_1-1} \sum_{j_2=0}^{n_2-1} |X_{(t_1^{(j_1+1)},t_2^{(j_2+1)}) }-X_{(t_1^{(j_1)},t_2^{(j_2+1)}) }-X_{(t_1^{(j_1+1)},t_2^{(j_2)}) }+X_{(t_1^{(j_1)},t_2^{(j_2)}) }| \label{eqn:3-6}\\
    &+ \sum_{j_1=0}^{n_1-1} \sum_{j_2=0}^{n_2-1} |X_{(t_1^{(j_1+1)},s_2^{(j_2+1)}) }-X_{(t_1^{(j_1)},s_2^{(j_2+1)}) }-X_{(t_1^{(j_1+1)},s_2^{(j_2)}) }+X_{(t_1^{(j_1)},s_2^{(j_2)}) }| \label{eqn:3-7}\\
    & + \sum_{j_1=0}^{n_1-1} \sum_{j_2=0}^{n_2-1} |X_{(s_1^{(j_1+1)},t_2^{(j_2+1)}) }-X_{(s_1^{(j_1)},t_2^{(j_2+1)}) }-X_{(s_1^{(j_1+1)},t_2^{(j_2)}) }+X_{(s_1^{(j_1)},t_2^{(j_2)}) }| \label{eqn:3-8}\\
    & + \sum_{j_1=0}^{n_1-1} \sum_{j_2=0}^{n_2-1} |X_{(s_1^{(j_1+1)},s_2^{(j_2+1)}) }-X_{(s_1^{(j_1)},s_2^{(j_2+1)}) }-X_{(s_1^{(j_1+1)},s_2^{(j_2)}) }+X_{(s_1^{(j_1)},s_2^{(j_2)}) }| \label{eqn:3-9}\\
    & + \sum_{j_1=0}^{n_1-1} |X_{(t_1^{(j_1+1)},t_2^{(0)}) }-X_{(t_1^{(j_1)},t_2^{(0)}) }-X_{(t_1^{(j_1+1)},s_2^{(0)}) }+X_{(t_1^{(j_1)},s_2^{(0)}) }| \label{eqn:3-10}\\
    & + \sum_{j_1=0}^{n_1-1} |X_{(s_1^{(j_1+1)},s_2^{(0)}) }-X_{(s_1^{(j_1)},s_2^{(0)}) }-X_{(s_1^{(j_1+1)},t_2^{(0)}) }+X_{(s_1^{(j_1)},t_2^{(0)}) }| \label{eqn:3-11}\\
    &+ \sum_{j_2=0}^{n_2-1} |X_{(t_1^{(0)},t_2^{(j_2+1)}) }-X_{(t_1^{(0)},t_2^{(j_2)}) }-X_{(s_1^{(0)},t_2^{(j_2+1)}) }+X_{(s_1^{(0)},t_2^{(j_2)}) }| \label{eqn:3-12}\\
    &+ \sum_{j_2=0}^{n_2-1} |X_{(s_1^{(0)},s_2^{(j_2+1)}) }-X_{(s_1^{(0)},s_2^{(j_2)}) }-X_{(t_1^{(0)},s_2^{(j_2+1)}) }+X_{(t_1^{(0)},s_2^{(j_2)}) }|. \label{eqn:3-13}
\end{align}

Let us assume from now that $\alpha \in [0,1)$, the case $\alpha=1$ requiring specific arguments developed later. Up to a permutation of the indices, we can assume $p_1 \geq p_2$. In this case, from \eqref{eqn:inesurdyadique}, we directly deduce, for all $\omega \in \Omega_\gamma$
\[ \eqref{eqn:3-5} \leq C_\gamma(\omega) 2^{-\gamma(p_1(1+ \alpha)+p_2(1-\alpha))}.\]
Similarly, for all $\omega \in \Omega_\gamma$, we bound \eqref{eqn:3-6}, \eqref{eqn:3-7}, \eqref{eqn:3-8} and \eqref{eqn:3-9} from above by
\begin{eqnarray*}
   & &C_\gamma(\omega) \sum_{j_1=0}^{n_1-1} \sum_{j_2=0}^{n_2-1} 2^{-\gamma(\max\{j_1+p_1+1,j_2+p_2+1\}(1+ \alpha)+\min\{j_1+p_1+1,j_2+p_2+1\}(1-\alpha))} \\
   & \leq & C_\gamma(\omega) \left( \sum_{j_1=0}^{+ \infty} \sum_{j_2 \leq j_1+p_1-p_2} 2^{-\gamma((j_1+p_1+1)(1+ \alpha)+(j_2+p_2+1)(1-\alpha))} \right. \\
   & &\quad  + \left. \sum_{j_1=0}^{+ \infty} \sum_{j_2 > j_1+p_1-p_2} 2^{-\gamma((j_2+p_2+1)(1+ \alpha)+(j_1+p_1+1)(1-\alpha))} \right) \\
   & \leq&  C_\gamma(\omega) \left(  2^{-\gamma(p_1 (1+\alpha) +p_2 (1-\alpha))} \sum_{j_1=0}^{+ \infty} \sum_{j_2=0}^{+ \infty} 2^{-\gamma(j_1 (1+\alpha) +j_2 (1-\alpha)+2)} \right. \\
   && \quad + \left. \sum_{j_1=0}^{+\infty} \sum_{j=1}^{+ \infty} 2^{-\gamma((j+j_1+p_1+1)(1+\alpha)+(j_1+p_1+1)(1-\alpha))}\right) \\
   & \leq & C_\gamma(\omega) ( c'  2^{-\gamma(p_1 (1+\alpha) +p_2 (1-\alpha))} + c'' 2^{- 2\gamma p_1} ) \\
   & \leq &C_\gamma(\omega) (c'+c'') 2^{-\gamma(p_1 (1+\alpha) +p_2 (1-\alpha))} .
\end{eqnarray*}
Using once again the inequality \eqref{eqn:inesurdyadique} on $\Omega_\gamma$, we bound \eqref{eqn:3-10} and \eqref{eqn:3-11} from above by
\begin{align*}
     C_\gamma(\omega) \sum_{j_1=0}^{+ \infty} 2^{-\gamma ((j_1+p_1+1)(1+\alpha)+p_2(1-\alpha))} \leq  C_\gamma(\omega) c' 2^{-\gamma(p_1 (1+\alpha) +p_2 (1-\alpha))}
\end{align*}
while we bound  \eqref{eqn:3-12} and \eqref{eqn:3-13} from above by
\begin{eqnarray*}
    &&C_\gamma(\omega) \sum_{j_2=0}^{n_2-1} 2^{-\gamma(\max\{p_1,j_2+p_2+1\}(1+ \alpha)+\min\{p_1,j_2+p_2+1\}(1-\alpha))} \\
    & \leq &C_\gamma(\omega) \left( \sum_{j_2 < p_1-p_2} 2^{-\gamma(p_1(1+\alpha)+(p_2+j_2+1)(1-\alpha))}  \right. \\
    &&\quad  + \left. \sum_{j_2 \geq p_1-p_2} 2^{-\gamma((p_2+j_2+1)(1+\alpha)+p_1(1-\alpha))} \right) \\
    & \leq&   C_\gamma(\omega) \left( c'2^{-\gamma(p_1 (1+\alpha) +p_2 (1-\alpha))}  + \sum_{j=0}^{+ \infty} 2^{-\gamma((p_1+j+1)(1+\alpha)+p_1(1-\alpha))} \right) \\
    & \leq & C_\gamma(\omega) (c'+c'') 2^{-\gamma(p_1 (1+\alpha) +p_2 (1-\alpha))}.
\end{eqnarray*}
In total, for all $\omega \in \Omega_\gamma$, we obtain
\begin{align}\label{eqn:accroirectdyadique}
    \eqref{eqn:kolmo:tobound} & \leq C_\gamma(\omega)'2^{-\gamma(p_1 (1+\alpha) +p_2 (1-\alpha))} \nonumber \\ 
    & \leq C_\gamma(\omega)'' \left(\max\{|t_1-s_1|,|t_2-s_2| \}^{1-\alpha}\min\{|t_1-s_1|,|t_2-s_2| \}^{1+\alpha}\right)^{\gamma}.
\end{align}
Note that, since we assume that the field $\{X_{(x_1,x_2)}\}_{(x_1,x_2) \in [0,1]^2}$ vanishes on the horizontal and vertical axes,  inequality \eqref{eqn:accroirectdyadique}  also implies
\begin{align}
    |X_{(t_1,s_2)}-X_{(s_1,s_2)}| & = |X_{(t_1,s_2)}-X_{(t_1,0)}-X_{(s_1,s_2)}+X_{(s_1,0)}| \nonumber \\
    & \leq C_\gamma(\omega)''  \left(\max\{|t_1-s_1|,|s_2| \}^{1-\alpha}\min\{|t_1-s_1|,|s_2| \}^{1+\alpha}\right)^{\gamma} \label{eqn:reguhorizon}
\end{align}
as well as
\begin{align}\label{eqn:reguverti}
       |X_{(s_1,t_2)}-X_{(s_1,s_2)}| \leq C_\gamma(\omega)'' \left(\max\{|s_1|,|t_2-s_2| \}^{1-\alpha}\min\{|s_1|,|t_2-s_2| \}^{1+\alpha}\right)^{\gamma}.
\end{align}

Let us now consider a sequence $(\gamma_n)_n$ with $\gamma_n \nearrow \frac{\beta}{\eta}$. The event
\[ \Omega_1 := \bigcap_n \Omega_{\gamma_n}\]
is of probability $1$ and the preceding argument shows that, on $\Omega_1$, the sample paths of $\{X_{(x_1,x_2)}\}_{(x_1,x_2) \in [0,1]^2}$ have the desired rectangular regularity property on $D^2$. Now, it remains to construct the version $\{\widetilde{X}_{(x_1,x_2)}\}_{(x_1,x_2) \in [0,1]^2}$ on $[0,1]^2$. 

First, on $\Omega_1$, we extent $X$ by using the density of $D^2$ in $[0,1]^2$ together with \eqref{eqn:reguhorizon} and \eqref{eqn:reguverti}. More precisely, if $(x_1,x_2) \in [0,1]$ and if $((x_1^{(j)},x_2^{(j)}))_j \subset D^2$ is such that $((x_1^{(j)},x_2^{(j)}))_j \to (x_1,x_2)$, then the inequalities \eqref{eqn:reguhorizon} and \eqref{eqn:reguverti} insure that $({X}_{(x_1^{(j)},x_2^{(j)})})_j$ is Cauchy. Hence,  we set $\widetilde{X}_{(x_1,x_2)}=\lim_j {X}_{(x_1^{(j)},x_2^{(j)})}$.  
Secondly, we set $\widetilde{X}_{(x_1,x_2)}=0$ on $\Omega_1^c$.

It is then clear that $\{\widetilde{X}_{(x_1,x_2)}\}_{(x_1,x_2) \in [0,1]^2}$ has the desired rectangular regularity property. To conclude,  it suffices now  to show that it is a version of $\{X _{(x_1,x_2)}\}_{(x_1,x_2) \in [0,1]^2}$. 
Of course, it suffices to work with $(x_1,x_2) \notin D^2$. Fix $\varepsilon>0$ and $j$. Using the triangle inequality, the fact that $\{X_{(x_1,x_2)}\}_{(x_1,x_2) \in [0,1]^2}$ vanishes on the horizontal and vertical axes, Markov's inequality and assumption \eqref{eqn:theoremkolmohyp}, we  get
\begin{align*}
    \PP(|{X}_{(x_1^{(j)},x_2^{(j)})}- X_{(x_1,x_2)}| \geq \varepsilon) & \leq \PP(|{X}_{(x_1^{(j)},x_2^{(j)})}- X_{(x_1^{(j)},x_2)}| \geq \frac{\varepsilon}{2})  + \PP(|{X}_{(x_1^{(j)},x_2)}- X_{(x_1,x_2)}| \geq \frac{\varepsilon}{2}) \\
    & \leq \PP(|{X}_{(x_1^{(j)},x_2^{(j)})}- X_{(x_1^{(j)},x_2)}-X_{(0,x_2^{(j)})}+X_{(0,x_2)}| \geq \frac{\varepsilon}{2}) \\
    &  \quad + \PP(|{X}_{(x_1^{(j)},x_2)}- X_{(x_1,x_2)}-X_{(x_1^{(j)},0)}+X_{(x_1,0)}| \geq \frac{\varepsilon}{2}) \\
    & \leq c 2^{\eta} \frac{\left(\max\{|x_1^{(j)}|,|x_2-x_2^{(j)}| \}^{1-\alpha}\min\{|x_1^{(j)}|,|x_2-x_2^{(j)}| \}^{1+\alpha}\right)^{1+\beta}}{\varepsilon^\eta} \\
    & \quad + c 2^{\eta}  \frac{\left(\max\{|x_1-x_1^{(j)}|,|x_2| \}^{1-\alpha}\min\{|x_1-x_1^{(j)}|,|x_2| \}^{1+\alpha}\right)^{1+\beta}}{\varepsilon^\eta}.
\end{align*}
This last inequality implies that the convergence ${X}_{(x_1^{(j)},x_2^{(j)})} \to X_{(x_1,x_2)}$ holds in probability and thus almost surely for a subsequence.  It directly follows that  $\PP(\widetilde{X}_{(x_1,x_2)}=X_{(x_1,x_2)})=1$.

\smallskip

To complete the proof, it remains to consider the case  $\alpha=1$. Using the same arguments, for any $\gamma \in [0, \frac{\beta}{\eta})$, we get
\[ \eqref{eqn:kolmo:tobound} \leq C_\gamma(\omega)' 2^{-2\gamma p_1} p_1 \leq   C_\gamma(\omega)'' \left(\min\{|t_1-s_1|,|t_2-s_2| \}\right)^{2 \gamma}  \left|\log\left(\min\{|t_1-s_1|,|t_2-s_2| \}\right)\right|\]
for all $\omega \in \Omega_\gamma$. In particular, for any $\gamma' < \gamma$, there is $C_{\gamma'}(\omega)>0$ such that
\[ \eqref{eqn:kolmo:tobound} \leq  C_{\gamma'}(\omega) \left(\min\{|t_1-s_1|,|t_2-s_2| \}\right)^{2 \gamma'},\]
which is sufficient to conclude the proof in exactly the same way as in the case $\alpha \in [0,1)$.
\end{proof}

\begin{remark}
 Let $\{X_{(x_1,x_2)}\}_{(x_1,x_2) \in \R^2}$ be a stochastic field vanishing on the axes and satisfying the inequality \eqref{eqn:theoremkolmo}. Then, for $x_1, x_2,h \in \R$, we have
 \begin{align*} |X_{(x_1+h,x_2)}-X_{(x_1,x_2)}| &= |X_{(x_1+h,x_2)}-X_{(x_1,x_2)}-X_{(x_1+h,0)} + X_{(x_1,0)}|\\
 & \leq  C\left(\max\{|h|,|x_2| \}^{1-\alpha}\min\{|h|,|x_2| \}^{1+\alpha}\right)^\gamma
 \end{align*}
 and, similarly
 \[ |X_{(x_1,x_2+h)}-X_{(x_1,x_2)}| \leq  C\left(\max\{|h|,|x_1| \}^{1-\alpha}\min\{|h|,|x_1| \}^{1+\alpha}\right)^\gamma.\]
 In particular, this means that the horizontal and vertical increments of the field are locally Hölder-continuous. Note that, in the case where the field $\{X_{(x_1,x_2)}\}_{(x_1,x_2) \in \R^2}$ satisfies the assumptions of  Theorem \ref{thm:kolmo}, this fact can also be observed by applying the (standard) Kolmogorov's continuity Theorem to  the increments of the field.
\end{remark}

Proposition \ref{prop:regularity} will now be easily obtained using the equivalence of Gaussian
moments: if $X$ is a Gaussian random variable then, for all $j \in \N$, we have
\[\E[X^{2j}] = \frac{(2j)!}{2^j j!} \E[X^2]^j.\]
Indeed, this relation allows us to state the following classical corollary of our version of Kolmogorov's continuity Theorem.

\begin{proposition}\label{regularity}
     Let $\alpha \in [0,1]$ be fixed. Assume that $\{X_{(x_1,x_2)}\}_{(x_1,x_2) \in \R^2}$ is a Gaussian stochastic field for which there exist some constants $\beta, c, R>0$ such that
  \[  \E[|\Delta  X_{(h_1,h_2);(x_1,x_2)}|^2] \leq c \left(\max\{|h_1|,|h_2| \}^{1-\alpha}\min\{|h_1|,|h_2| \}^{1+\alpha}\right)^{\beta}\]
  for all $(x_1,x_2)\in \R^2$ and $(h_1,h_2) \in \R^2$ with $|h_1| < R$ and $|h_2|<R$. Then, there exists a version $\{\widetilde{X}_{(x_1,x_2)}\}_{(x_1,x_2) \in \R^2}$ of $\{X_{(x_1,x_2)}\}_{(x_1,x_2) \in \R^2}$ for which, for all $\gamma \in [0,\frac{\beta}{2})$ and for every bounded intervals $I,J$ of $\mathbb{R}$, for all $\omega \in \Omega$, there exists a finite constant $C(\omega)>0$ such that
    \[
        |\Delta  \widetilde{X}_{(h_1,h_2);(x_1,x_2)}| \leq C(\omega)\left(\max\{|h_1|,|h_2| \}^{1-\alpha}\min\{|h_1|,|h_2| \}^{1+\alpha}\right)^\gamma
    \]
for every $x_1 \in I$, $x_2 \in J$ and $h_1, h_2 \in \mathbb{R}$ with $x_1+h_1 \in I$ and $x_2 + h_2 \in J$.
\end{proposition}

\section{Irregularities of the trajectories}


The aim of this section is to prove that the regularity of the rectangular increments of the WTFBFs, as obtained in Proposition \ref{prop:regularity}  as a consequence of the Kolmogorov continuity-type Theorem \ref{thm:kolmo}, is optimal.

\begin{theorem}\label{thm:irreg}
Fix $\alpha \in [0,1]$ and $H \in (0,1)$.    Almost surely, for every bounded intervals $I,J$ of $\R$, one has
$$
 \sup_{\stackrel[h_1 \neq 0, h_2 \neq 0]{(x_1,x_2),  (x_1+h_1, x_2+h_2) \in I \times J }{}}  \frac{|\Delta  X^{\alpha,H}_{(h_1,h_2);(x_1,x_2)}| }{\left(\max\{|h_1|,|h_2| \}^{1-\alpha}\min\{|h_1|,|h_2| \}^{1+\alpha}\right)^{H}} = + \infty.
$$
\end{theorem}

The proof of Theorem \ref{thm:irreg} is based on several lemmas, that rely on estimating the size  of the so-called ``hyperbolic wavelet coefficients'' of our fields, see Section \ref{sec:space} for more details about hyperbolic wavelet basis. In what follows, we will denote by $\{2^{j/2} \psi(2^j \cdot -k)  \, : \, (j,k) \in \Z^2 \}$  the Lemari\'e-Meyer orthonormal wavelet basis of the Hilbert space $L^2(\R)$, introduced in \cite{LM86}. Its particular features include the fact that the mother wavelet $\psi$ belongs to the Schwartz class of $C^\infty$ functions whose derivatives of all orders decay rapidly, that it has vanishing moments of every order and that 
\begin{equation}\label{eq:support}
\textrm{supp } \widehat{\psi} \subseteq \left[-\f{8\pi}{3},-\f{2\pi}{3}\right] \cup \left[\f{2\pi}{3},\f{8\pi}{3}\right] . 
\end{equation}
For every $(j_1,j_2), (k_1,k_2) \in \N^2$, we will use the compact notation
$$
\bj = (j_1,j_2) \quad \text{and} \quad \bk = (k_1,k_2)
$$
together with
$$
\max (\bj) = \max \{j_1,j_2 \}  \quad \text{and} \quad \min (\bj) = \min \{j_1,j_2 \}. 
$$
Let  $ c_{\bj,\bk}$, $(\bj,\bk) \in \Z^2$, denote the hyperbolic wavelet coefficients of  $X^{\alpha, H}$ in the the Lemari\'e-Meyer basis, defined by 
\begin{equation}\label{WC}
     c_{\bj,\bk}
     = 2^{j_1+j_2} \int_{\R^2} X^{\alpha,H}_{(x_1,x_2)} \psi(2^{j_1}x_1-k_1)\psi(2^{j_2}x_2-k_2) d\boldsymbol{x} .  
\end{equation}
We will show that these coefficients are independent as soon as they correspond to distant scales, and controlling their second-order moment will allow us to estimate their size as the scale increases. The vanishing moment of $\psi$  will then enable us to obtain information about the rectangular increments of the field. Our first objective will be to prove that  the wavelet coefficients given in Equation \eqref{WC}  are well-defined. Let us begin with the following lemma, which provides an estimate of the covariance of the rectangular increments of the process  $X^{\alpha, H}$. From now on, we assume that  $\alpha \in [0,1]$ and $H \in (0,1)$ are fixed. 

\begin{lemma}\label{lem:irreg3}
One has 
\begin{align*} 
&\big| \mathbb{E}\big[ \Delta  X^{\alpha,H}_{(h_1,h_2);(x_1,x_2)}\overline{\Delta   X^{\alpha,H}_{(\ell_1,\ell_2);(y_1,y_2)}}\big]  \big| \\[2ex]
& \quad \quad  \quad   \leq  \int_{\R^2}\frac{\min \{2, |h_1\xi_1| \}   \min \{2, |h_2\xi_2| \}   \min \{2, |\ell_1\xi_1| \}    \min \{2, |\ell_2\xi_2| \} | }{\big(\phi_{\alpha,H}(\xi_1,\xi_2)\big)^2 }d\boldsymbol{\xi}
\end{align*}
for all $(h_1,h_2),(\ell_1, \ell_2),(x_1,x_2),(y_1,y_2)  \in \mathbb{R}^2$. 
\end{lemma}

\begin{proof}
    By definition of $X^{\alpha, H}$, one has
$$
        \Delta  X_{(h_1,h_2); (x_1,x_2)}^{\alpha,H} =  \int_{\mathbb{R}^2} e^{i(x_1 \xi_1 + x_2 \xi_2)}\frac{(e^{i h_1 \xi_1}-1) (e^{i h_2 \xi_2}-1)}{\phi_{\alpha,H}(\xi_1,\xi_2)}  d\hat{\bf{W}}({\boldsymbol{\xi}}) .
$$
The isometry property of the Wiener integral gives then
\begin{align*}
   & \mathbb{E}\big[ \Delta  X^{\alpha,H}_{(h_1,h_2);(x_1,x_2)}\overline{\Delta  X^{\alpha,H}_{(h_1,h_2);(y_1,y_2)}}\big]\\[1.5ex]
   & = \int_{\R^2}\frac{e^{i((x_1-y_1) \xi_1 + (x_2-y_2) \xi_2)}(e^{i h_1 \xi_1}-1) (e^{i h_2 \xi_2}-1)(e^{-i \ell_1 \xi_1}-1) (e^{-i \ell_2 \xi_2}-1)}{\big(\phi_{\alpha,H}(\xi_1,\xi_2)\big)^2 }d\boldsymbol{\xi} .
\end{align*}
It implies that 
\begin{align*}
   & \big|\mathbb{E}\big[ \Delta  X^{\alpha,H}_{(h_1,h_2);(x_1,x_2)}\overline{\Delta  X^{\alpha,H}_{(h_1,h_2);(y_1,y_2)}}\big] \big |\\[1.5ex]
   & \leq  \int_{\R^2}\frac{|e^{i h_1 \xi_1}-1| \,  |e^{i h_2 \xi_2}-1| \, |e^{-i \ell_1 \xi_1}-1| \, |e^{-i \ell_2 \xi_2}-1| }{\big(\phi_{\alpha,H}(\xi_1,\xi_2)\big)^2 }d\boldsymbol{\xi}  \\
   & \leq  \int_{\R^2}\frac{\max \{2, |h_1\xi_1| \} \,  \max \{2, |h_2\xi_2| \} \,  \max \{2, |\ell_1\xi_1| \}  \,  \max \{2, |\ell_2\xi_2| \} | }{\big(\phi_{\alpha,H}(\xi_1,\xi_2)\big)^2 }d\boldsymbol{\xi} 
\end{align*}
by noticing that
$$
|e^{i\xi}-1| = 2 \, |\sin (\frac{\xi}{2}) | \leq \min \{2, |\xi| \}.
$$
    \end{proof}

\begin{lemma}\label{lem:irreg4}
Let us fix $L\in (0,1]$ and let us consider for every $(\bj,\bk) \in \N \times \Z$ the coefficient
$$
d_{\bj,\bk}
= \int_{S_{\bj}} \Delta  X^{\alpha,H}_{(\frac{y_1}{2^{j_1}},\frac{y_2}{2^{j_2}});(\frac{k_1}{2^{j_1}},\frac{k_2}{2^{j_2}})}  \psi(y_1)\psi(y_2) d\boldsymbol{y} 
$$
where
$$
S_{\bj} = \big\{(y_1,y_2) \in \R^2 : |y_1|\geq 2^{j_1}L \text{ or } |y_2| \geq  2^{j_2}L   \big\}.
$$
 For every $N \geq 3$, there exists $C>0$ such that 
$$
\mathbb{E}\big[ |d_{\bj,\bk} |^2\big] \leq C \, 2^{-2N \min \{j_1,j_2 \}}
$$
for all $(\bj,\bk) \in \N \times \Z$.
\end{lemma}

\begin{proof}
For the first part, note that Lemma  \ref{lem:irreg3} and Fubini theorem give
\begin{align*}
   & \mathbb{E}\big[ |d_{\bj,\bk} |^2\big] \\ 
    & \leq  \int_{ S_{\bj}} \int_{ S_{\bj}} \left|\mathbb{E}\left[\Delta X^{\alpha,H}_{(\frac{y_1}{2^{j_1}},\frac{y_2}{2^{j_2}});(\frac{k_1}{2^{j_1}},\frac{k_2}{2^{j_2}})} \overline{ \Delta X^{\alpha,H}_{(\frac{x_1}{2^{j_1}},\frac{x_2}{2^{j_2}});(\frac{k_1}{2^{j_1}},\frac{k_2}{2^{j_2}})}} \right]\right| \, \big|\psi(y_1)\psi(y_2)\overline{ \psi(x_1)\psi(x_2)}\big|d\boldsymbol{y} d\boldsymbol{x}\\[2ex]
     & \leq \int_{\R^2} \!\int_{ S_{\bj}} \!\int_{ S_{\bj}}\!\! \frac{\min \{2, |\frac{y_1}{2^{j_1}}\xi_1| \}   \min \{2, |\frac{y_2}{2^{j_2}}\xi_2| \}   \min \{2, |\frac{x_1}{2^{j_1}}\xi_1| \}    \min \{2, |\frac{x_2}{2^{j_2}}\xi_2| \} | }{\big(\phi_{\alpha,H}(\xi_1,\xi_2)\big)^2 }\\
    & \hspace{7cm} \times \big|\psi(y_1)\psi(y_2)\overline{ \psi(x_1)\psi(x_2)}\big|d\boldsymbol{y} d\boldsymbol{x} d\boldsymbol{\xi}\\[2ex]
      & \leq 2^{2(j_1+j_2)} \int_{\R^2}\int_{ S} \int_{ S}\frac{\min \{2, |u_1\xi_1| \}   \min \{2, |u_2\xi_2| \}   \min \{2, |v_1\xi_1| \}    \min \{2, |v_2\xi_2| \} | }{\big(\phi_{\alpha,H}(\xi_1,\xi_2)\big)^2 }\\
    & \hspace{7cm} \times \big|\psi(2^{j_1}u_1)\psi(2^{j_2}u_2)\overline{ \psi(2^{j_1}v_1)\psi(2^{j_2}v_2)}\big|d\boldsymbol{u}d\boldsymbol{v} d\boldsymbol{\xi}  \\[2ex]
    & \leq C 2^{2(j_1+j_2)}  \int_{ S} \int_{ S}\max \{1, |u_1| \}   \max \{1, |u_2| \}   \max \{1, |v_1| \}    \max \{1, |v_2| \} |  \\
    & \hspace{7cm} \times \big|\psi(2^{j_1}u_1)\psi(2^{j_2}u_2)\overline{ \psi(2^{j_1}v_1)\psi(2^{j_2}v_2)}\big|d\boldsymbol{u}d\boldsymbol{v} \\[2ex]
    & = C 2^{2(j_1+j_2)}  \left(\int_{ S} \max \{1, |u_1| \} \max \{1, |u_2| \} \big|\psi(2^{j_1}u_1)\psi(2^{j_2}u_2)| d\boldsymbol{u}\right)^2  
\end{align*}
where  $S= S_{(0,0)}$, 
$$
C=\int_{\R^2} \frac{\min \{2, |\xi_1| \}   \min \{2, |\xi_2| \}   \min \{2, |\xi_1| \}    \min \{2, |\xi_2| \}  }{\big(\phi_{\alpha,H}(\xi_1,\xi_2)\big)^2 }d\boldsymbol{\xi} , 
$$
and by using the change of variables $u_1 = 2^{-j_1}y_1$, $u_2= 2^{-j_2}y_2$, $v_1 = 2^{-j_1}x_1$, $v_2= 2^{-j_2}x_2$ and the relation
$$
\min\{ 2, |w \xi|\} \le \min\{ 2, |\xi|\} \max\{ 1, |w |\} 
$$
for all $w, \xi \in \R$.  Note that the constant $C$ is finite by Lemma \ref{lem:biendefini}. We will now decompose the integral over $S$ in three parts, corresponding to the sets
\begin{align*}
S^{(1)} & = \big\{(u_1,u_2) \in \R^2 : |u_1|\geq L \text{ and } |u_2| \geq  L   \big\}\\
S^{(2)} & = \big\{(u_1,u_2) \in \R^2 : |u_1|\geq L \text{ and } |u_2| <  L   \big\}\\
S^{(3)} & = \big\{(u_1,u_2) \in \R^2 : |u_1|< L \text{ and } |u_2| \geq  L   \big\}\\
\end{align*}
and give an upper bound for each term by using the fast decay of the wavelet
$$
|\psi(t) | \leq \frac{C_{2N}}{(1+|t|)^{2N}} \quad \forall t \in \R. 
$$
First, one has
\begin{eqnarray*}
&& \int_{ S^{(1)}} \max \{1, |u_1| \} \max \{1, |u_2| \} \big|\psi(2^{j_1}u_1)\psi(2^{j_2}u_2)| d\boldsymbol{u}\\
&\leq & \int_{ S^{(1)}}  \max \{1, |u_1| \} \max \{1, |u_2| \}\frac{C^2_{2N}}{(1+|2^{j_1}u_1|)^{2N}(1+|2^{j_2}u_2|)^{2N}}  d\boldsymbol{u}\\[1ex]
&\leq & C^2_{2N} 2^{-N(j_1+j_2)}\int_{ S^{(1)}} \frac{ \max \{1, |u_1| \} \max \{1, |u_2| \}}{(1+|2^{j_1}u_1|)^{N}(1+|2^{j_2}u_2|)^{N}}  d\boldsymbol{u}\\[1ex]
& \leq & C^2_{2N} 2^{-N(j_1+j_2)}2^{-(j_1+j_2)}\int_{|t_1| \geq 2^{j_1}L} \frac{\max\{ 1, 2^{-j_1}|t_1|\}}{(1+|t_1|)^{N}} dt_1\int_{|t_2| \geq 2^{j_2}L}  \frac{\max\{ 1, 2^{-j_2}|t_2|\}}{(1+|t_2|)^{N}} dt_2\\
& \leq & C^2_{2N} 2^{-N(j_1+j_2)}2^{-(j_1+j_2)}\left(\int_{\R} \frac{\max\{ 1, |t|\}}{(1+|t|)^{N}} dt \right)^2
\end{eqnarray*}
by using the change of variable $t_1=2^{j_1u_1}$, $t_2=2^{j_2}t_2$. For the second term, we use as before the fast decay of the wavelet for the part corresponding to $|u_1|\geq L$ while for the part corresponding to $|u_2| <  L  $, we use the fact that $ \max \{1, |u_2| \}  = 1$ since $L\leq 1$. More precisely, we obtain
\begin{eqnarray*}
  && \int_{ S^{(2)}} \max \{1, |u_1| \} \max \{1, |u_2| \} \big|\psi(2^{j_1}u_1)\psi(2^{j_2}u_2)| d\boldsymbol{u}\\  
  &\leq & \left(\int_{ |u_1|\geq L}  \max \{1, |u_1| \} \frac{C_{2N}}{(1+|2^{j_1}u_1|)^{2N}} du_1 \right)\left(\int_{ |u_2|< L} \big|\psi(2^{j_2}u_2)| du_2 \right)\\[1ex]
  & \leq & C^2_{2N} 2^{-Nj_1}2^{-j_1} \left(\int_{\R} \frac{\max\{ 1, |t_1|\}}{(1+|t_1|)^{N}} dt_1 \right) 2^{-j_2}\left(\int_{ \R} \big|\psi(t_2)| dt_2 \right).
\end{eqnarray*}
We proceed in the same way for the integral over $S^{(3)}$,  by switching the roles of $j_1$ and $j_2$. Putting everything together, it follows that there exists a constant $C'>0$ such that 
\begin{eqnarray*}
 \left|\mathbb{E}\big[ |d_{\bj,\bk} |^2\big]\right|
  & \leq & C 2^{2(j_1+j_2)}  \left(\sum_{m=1}^3\int_{ S^{(m)}} \max \{1, |u_1| \} \max \{1, |u_2| \} \big|\psi(2^{j_1}u_1)\psi(2^{j_2}u_2)| d\boldsymbol{u}\right)^2  \\
  & \leq & C' 2^{2(j_1+j_2)} \left(2^{-N(j_1+j_2)}2^{-(j_1+j_2)} + 2^{-Nj_1}2^{-(j_1+j_2)} + 2^{-Nj_2}2^{-(j_1+j_2)}\right)^2 \\
 & \leq & 9 C' 2^{-2 N \min \{j_1,j_2 \}},
   \end{eqnarray*}
   which gives the conclusion.
\end{proof}

A direct consequence of the previous Lemma is that, almost surely, for every $(\bj,\bk)$, the wavelet coefficient $c_{\bj,\bk}$ given in Equation \eqref{WC} is well-defined. Indeed, by using the vanishing moment of the wavelet and the change of variables $y_1=2^{j_1}x_1-k_1$, $y_2= 2^{j_2}x_2-k_2$, we can write
\begin{align}\label{WC2}
    c_{\bj,\bk}& 
    =  \int_{\R^2} X^{\alpha,H}_{(\frac{y_1+k_1}{2^{j_1}},\frac{y_2+k_2}{2^{j_2}})} \psi(y_1)\psi(y_2)d\boldsymbol{y} \nonumber\\
    & = \int_{\R^2} \Delta  X^{\alpha,H}_{(\frac{y_1}{2^{j_1}},\frac{y_2}{2^{j_2}});(\frac{k_1}{2^{j_1}},\frac{k_2}{2^{j_2}})}  \psi(y_1)\psi(y_2) d\boldsymbol{y} \nonumber \\
    & = \int_{R_{\bj}} \Delta  X^{\alpha,H}_{(\frac{y_1}{2^{j_1}},\frac{y_2}{2^{j_2}});(\frac{k_1}{2^{j_1}},\frac{k_2}{2^{j_2}})}  \psi(y_1)\psi(y_2) d\boldsymbol{y}
    +\int_{S_{\bj}} \Delta  X^{\alpha,H}_{(\frac{y_1}{2^{j_1}},\frac{y_2}{2^{j_2}});(\frac{k_1}{2^{j_1}},\frac{k_2}{2^{j_2}})}  \psi(y_1)\psi(y_2) d\boldsymbol{y}
\end{align}
where $R_{\bj}= \R^2 \setminus S_{\bj}$. The first integral is almost surely well-defined thanks to Proposition \ref{regularity}, while the second integral is almost surely finite thanks to Lemma \ref{lem:irreg4}.

\begin{lemma}\label{lem:irreg1}
Almost surely, for every  $(\bj,\bk)$, one has
$$ c_{\bj,\bk} = \int_{\R^2}\frac{e^{i(k_12^{-j_1}\xi_1 + k_22^{-j_2}\xi_2)}\widehat{\psi} (2^{-j_1}\xi_1)\widehat{\psi} (2^{-j_2}\xi_2)}{\phi_{\alpha,H}(\xi_1,\xi_2)}  d\hat{\bf{W}}({\boldsymbol{\xi}}).$$
In particular,
\begin{enumerate}
    \item if $\|\bj-\bj'\|_{\infty} >1$, the wavelet coefficients $c_{\bj,\bk}$ and $c_{\bj',\bk'}$ are independent,
    \item there exist two constants $c,d>0$ such that 
    $$
   c \, 2^{-2 (\max (\bj) H_\alpha^+ + \min (\bj) H_\alpha^-  )} \leq \E[\, |c_{\bj,\bk}|^2\, ] \leq d \, 2^{-2 (\max (\bj) H_\alpha^+ + \min (\bj) H_\alpha^-  )}
    $$
    for all $\bj,\bk$. 
\end{enumerate}
\end{lemma}

\begin{proof}
We have
\begin{align*}
    c_{\bj,\bk}
    & = \int_{\R^2} \Delta  X^{\alpha,H}_{(\frac{y_1}{2^{j_1}},\frac{y_2}{2^{j_2}});(\frac{k_1}{2^{j_1}},\frac{k_2}{2^{j_2}})}  \psi(y_1)\psi(y_2) d\boldsymbol{y}\\[1ex]
    & = \int_{\R^2}\int_{\mathbb{R}^2}  \frac{(e^{i (y_1+k_1)2^{-j_1} \xi_1}-e^{i k_1 2^{-j_1} \xi_1}) (e^{i (y_2+k_2)2^{-j_2} \xi_2}-e^{i k_2 2^{-j_2} \xi_2})}{\phi_{\alpha,H}(\xi_1,\xi_2)} \psi(y_1)\psi(y_2) d\hat{\bf{W}}({\boldsymbol{\xi}})  d\boldsymbol{y} \\[1ex]
    & = \int_{\R^2}\frac{e^{i(k_12^{-j_1}\xi_1 + k_22^{-j_2}\xi_2)}}{\phi_{\alpha,H}(\xi_1,\xi_2)}\left(\int_{\mathbb{R}}  (e^{i y_12^{-j_1} \xi_1}-1) \psi(y_1) dy_1 \right) \left(\int_{\R }(e^{i y_2 2^{-j_2} \xi_2}-1) \psi(y_2)   dy_2\right)  d\hat{\bf{W}}({\boldsymbol{\xi}}) \\[1ex]
    & = \int_{\R^2}\frac{e^{i(k_12^{-j_1}\xi_1 + k_22^{-j_2}\xi_2)}\widehat{\psi} (2^{-j_1}\xi_1)\widehat{\psi} (2^{-j_2}\xi_2)}{\phi_{\alpha,H}(\xi_1,\xi_2)}  d\hat{\bf{W}}({\boldsymbol{\xi}}) 
\end{align*}
using a Fubini type argument for Wiener integral, see e.g. \cite[Lemma 2.10.]{MR1617045}. 
It gives the first part of the lemma. For the second part,  the isometry property of the Wiener integral gives 
\begin{align*}
    \mathbb{E}\big[c_{\bj,\bk} \overline{c_{\bj',\bk'}}\big]
    & = \int_{\R^2 }\frac{e^{i(k_12^{-j_1}-k'_12^{-j'_1})\xi_1} e^{i(k_22^{-j_2}-k'_22^{-j'_2})\xi_2}\widehat{\psi} (2^{-j_1}\xi_1)\overline{\widehat{\psi}} (2^{-j'_1}\xi_1) \widehat{\psi} (2^{-j_2}\xi_2)\overline{\widehat{\psi}} (2^{-j'_2}\xi_2)}{\big(\phi_{\alpha,H}(\xi_1,\xi_2)\big)^2}  d\boldsymbol{\xi} . 
\end{align*}
Assume now for example that $j_1 - j'_1>1$.  Using the localization of $\hat{\psi}$ given in Equation \eqref{eq:support},  if $2^{-j_1} \xi_1 \in \textrm{supp } \widehat{\psi}$, then we have
$$
|\xi_1|2^{-j'_1} = 2^{-j_1}|\xi_1 | 2^{j_1-j'_1} \geq \frac{2\pi}{3} 2^{j_1-j'_1} \geq \frac{8\pi}{3}.
$$
Consequently, we obtain $\widehat{\psi} (2^{-j_1}\xi_1) \widehat{\psi} (2^{-j'_1}\xi_1)=0$ for all $\xi \in \R$ and it follows that the coefficients $c_{\bj,\bk}$ and ${c_{\bj',\bk'}}$ are independent, thanks to the Gaussianity of the wavelet coefficients. The other cases are treated in the same way.

Finally, one also has
$$
     \mathbb{E}\big[|c_{\bj,\bk} |^2\big]
     = \int_{\R^2 }\frac{|\widehat{\psi} (2^{-j_1}\xi_1)|^2 |\widehat{\psi} (2^{-j_2}\xi_2)|^2}{\big(\phi_{\alpha,H}(\xi_1,\xi_2)\big)^2}  d\boldsymbol{\xi}
     = 2^{j_1+j_2}\int_{\R^2 }\frac{|\widehat{\psi} (\eta_1)|^2 |\widehat{\psi} (\eta_2)|^2}{\big(\phi_{\alpha,H}(2^{j_1}\eta_1,2^{j_2}\eta_2)\big)^2}  d\boldsymbol{\eta}
$$
by using the change of variables $\eta_1 = 2^{-j_1} \xi_1$, $\eta_2 = 2^{-j_2} \xi_2$. Assume now that $j_1 - j_2>1$. Then, if $\eta_1 , \eta_2 \in \textrm{supp } \widehat{\psi}$, one has 
$$
|2^{j_1} \eta_1| \geq 2^{j_1} \frac{2\pi}{3} \geq 2^{j_2}\frac{8\pi}{3}  \geq |2^{j_2}\eta_2|
$$
hence
$$
\big(\phi_{\alpha,H}(2^{j_1}\eta_1,2^{j_2}\eta_2)\big)^2= \frac{2^{-j_1(2 H_\alpha^+ +1)}2^{-j_2 (2H_\alpha^- +1) }}{|\eta_2|^{2H_\alpha^- +1} |\eta_1 |^{2H_\alpha^++1 }},
$$
see \eqref{eq:H+H-} for the definition of $H_\alpha^+$ and $H_\alpha^-$. 
We obtain directly that 
$$
\mathbb{E}\big[|c_{\bj,\bk} |^2\big] 
= 2^{-2 j_1 H_\alpha^+ - 2j_2 H_\alpha^-  }\int_{\R }\frac{|\widehat{\psi} (\eta)|^2 }{|\eta |^{2H_\alpha^++1 }} d\eta\int_{\R }\frac{|\widehat{\psi} (\eta)|^2 }{|\eta |^{2H_\alpha^-+1 }} d\eta.
$$
A similar argument for $j_2- j_1>1$ gives
$$
\mathbb{E}\big[|c_{\bj,\bk} |^2\big] = c_1 2^{-2 (\max (\bj) H_\alpha^+ + \min (\bj) H_\alpha^-  )} \quad \text{ if } |j_1-j_2| >1 
$$
for 
$$c_1= \int_{\R }\frac{|\widehat{\psi} (\eta)|^2 }{|\eta |^{2H_\alpha^++1 }} d\eta\int_{\R }\frac{|\widehat{\psi} (\eta)|^2 }{|\eta |^{2H_\alpha^-+1 }} d\eta. $$
It remains to treat the case where $|j_1-j_2| \leq 1$. First, assume that $j_1=j_2$. Then 
\begin{align*}
     \mathbb{E}\big[|c_{\bj,\bk} |^2\big]
    & = 2^{2j_1}\int_{\R^2 }\frac{|\widehat{\psi} (\eta_1)|^2 |\widehat{\psi} (\eta_2)|^2}{\big(\phi_{\alpha,H}(2^{j_1}\eta_1,2^{j_1}\eta_2)\big)^2}  d\eta_1d\eta_2 \\
    & = 2^{-2(j_1 H^+_\alpha + j_1 H^-_\alpha)} \int_{\R^2} \frac{|\widehat{\psi} (\eta_1)|^2 |\widehat{\psi} (\eta_2)|^2}{\big(\phi_{\alpha,H}(\eta_1,\eta_2)\big)^2}  d\eta_1d\eta_2\\
    & = c_2 2^{-2 (\max (\bj) H_\alpha^+ + \min (\bj) H_\alpha^-  )}
\end{align*}
with
$$
c_2=\int_{\R^2} \frac{|\widehat{\psi} (\eta_1)|^2 |\widehat{\psi} (\eta_2)|^2}{\big(\phi_{\alpha,H}(\eta_1,\eta_2)\big)^2}  d\eta_1d\eta_2. 
$$
The cases $j_1=j_2\pm 1$ are treated similarly  with the two constants
$$
c_3 =\int_{\R^2} \frac{|\widehat{\psi} (\eta_1)|^2 |\widehat{\psi} (\eta_2)|^2}{\big(\phi_{\alpha,H}(2\eta_1,\eta_2)\big)^2}  d\eta_1d\eta_2 
\quad \text{and} \quad 
c_4= \int_{\R^2} \frac{|\widehat{\psi} (\eta_1)|^2 |\widehat{\psi} (\eta_2)|^2}{\big(\phi_{\alpha,H}(\eta_1,2\eta_2)\big)^2}  d\eta_1d\eta_2.
$$
It suffices then to take $c= \min \{ c_1,c_2,c_3,c_4\}$ and $d= \max \{ c_1,c_2,c_3,c_4\}$. 
\end{proof}

The following final  classical lemma provides the asymptotic behavior of a sequence of (independent) standard Gaussian random variables. It is a direct consequence of the Borel-Cantelli lemma together with the classical estimate of the tail behavior of a standard Gaussian random variable $Z$, given by
$$
\lim_{x \to + \infty} \frac{\PP(|Z|>\varepsilon)}{x^{-1} e^{-x^2/2}} = \sqrt{\frac{2}{\pi}}. 
$$

\begin{lemma}\label{lem:irreg2}
Assume that $(Z_n)_{n \in \N}$ is a sequence of $\mathcal{N}(0,1)$ random variables.
\begin{enumerate}
    \item Almost surely, one has 
    $$
    \limsup_{n \to + \infty} \frac{|Z_n|}{ \sqrt{\log(n+2)}} < + \infty . 
    $$
    \item If the random variables $Z_n$, $n \in \N$, are independent, then one has almost surely
    $$
    \limsup_{n \to + \infty} \frac{|Z_n|}{ \sqrt{\log(n+2)}} >0. 
    $$
\end{enumerate}
\end{lemma}

We are now ready to prove the main result of this section.

\begin{proof}[Proof of Theorem \ref{thm:irreg}]
Fix a scale $J$ large enough so that  there is $(K_1,K_2)\in \Z^2 $ satisfying
$$
\frac{K_1}{2^J} \in I^\circ \quad \text{and} \quad  \frac{K_2}{2^{J+2}} \in J^{\circ} .
$$
Let us also fix $L\in (0,1]$ such that 
$$
\left( \frac{K_1}{2^J} -L, \frac{K_1}{2^J} + L \right) \subset I \quad \text{and} \quad \left( \frac{K_2}{2^{J +2}} -L, \frac{K_2}{2^{J+2}} + L \right) \subset J. 
$$
For every $j \geq J$, we denote by $(\bj,\bk)$ the unique couple satisfying $j_1 = j = j_2-2$ and $\frac{K_1}{2^J} =\frac{k_1}{2^j} $, $\frac{K_2}{2^{J+2}} = \frac{k_2}{2^{j_2}}$. As done in Equation \eqref{WC2}, we know that 
\begin{align*}
 c_{\bj,\bk}
   & = \int_{\R^2} \Delta  X^{\alpha,H}_{(\frac{y_1}{2^{j_1}},\frac{y_2}{2^{j_2}});(\frac{k_1}{2^{j_1}},\frac{k_2}{2^{j_2}})}  \psi(y_1)\psi(y_2) d\boldsymbol{y}\\
&    = \int_{R_{\bj}} \Delta  X^{\alpha,H}_{(\frac{y_1}{2^{j_1}},\frac{y_2}{2^{j_2}});(\frac{k_1}{2^{j_1}},\frac{k_2}{2^{j_2}})}  \psi(y_1)\psi(y_2) d\boldsymbol{y}
    +\int_{S_{\bj}} \Delta  X^{\alpha,H}_{(\frac{y_1}{2^{j_1}},\frac{y_2}{2^{j_2}});(\frac{k_1}{2^{j_1}},\frac{k_2}{2^{j_2}})}  \psi(y_1)\psi(y_2) d\boldsymbol{y}\\
    & := \tilde{d}_{\bj,\bk} + d_{\bj,\bk}
\end{align*}
 Lemma \ref{lem:irreg4} together with Lemma \ref{lem:irreg2}
implies that there is an event of probability one such that
$$
\limsup_{j_1 \to + \infty }\frac{|d_{\bj,\bk} | }{2^{-2Nj_1} \sqrt{\log(j_1+2)}}< + \infty 
$$
since $\min \{j_1,j_2\}= j_1$. Hence, there is a constant $C_0>0$ such that
$$
|d_{\bj,\bk}| \leq C_0 2^{-N(2j_1 +2)} \sqrt{\log(j_1+2)}
$$
for every $\bj,\bk$.
Now, let us assume by contradiction that on this event of probability one,  there exists a constant $D>0$ such that
$$
\sup_{(x_1,x_2), (x_1+h_1, x_2+h_2) \in I \times J }  |\Delta  X^{\alpha,H}_{(h_1,h_2);(x_1,x_2)}| \leq D \left(\max\{|h_1|,|h_2| \}^{1-\alpha}\min\{|h_1|,|h_2| \}^{1+\alpha}\right)^{H}. 
$$
We obtain then that
\begin{align*}
|\tilde{d}_{\bj,\bk } | & \leq D \int_{R_{j,k}} \left(\max\{|2^{-j_1}y_1|,|2^{-j_2}y_2| \}^{1-\alpha}\min\{|2^{-j_1}y_1|,|2^{-j_2}y_2| \}^{1+\alpha}\right)^{H} |\psi(y_1) \psi(y_2)| d\boldsymbol{y} \\
& \leq 2^{-2 j_1 H}D \int_{R_{j,k}} \left(\max\{|y_1|,|\frac{y_2}{4}| \}^{1-\alpha}\min\{|y_1|,|\frac{y_2}{4}| \}^{1+\alpha}\right)^{H} |\psi(y_1) \psi(y_2)| d\boldsymbol{y} \\
& = C 2^{-2 j_1 H }
\end{align*}
for some constant $C>0$. 
Now, Lemma \ref{lem:irreg1}  and Lemma \ref{lem:irreg2}, applied to scales sufficiently far apart to ensure independence, imply that there are  a constant $C'>0$ and infinitely many  scales such that 
$$
|c_{\bj,\bk}|\geq C' \sqrt{\log(j_1+2)} 2^{-(\max(\bj) H^+_\alpha + \min (\bj) H^-_\alpha)} = C'' \sqrt{\log(j_1+2)} 2^{ -2 j_1 H}
$$
since $j_2=j_1 + 2$.
Putting everything together, we obtain that for infinitely many scales,
$$
C'' \sqrt{\log(j_1+2)} 2^{ - 2 j_1 H} \leq |c_{\bj,\bk}| \leq  C 2^{-2 j_1 H } + C_0 2^{-N(2j_1 +2)} \sqrt{\log(j_1+2)}
$$
which is impossible. It gives the conclusion.  
\end{proof}

\begin{remark}
The same arguments show that the irregularity can  actually  be observed in any direction $j_2 =j_1 +p$, as soon as $p$ is an integer chosen such that $p \notin \{-1,0,1 \}$.
\end{remark}

\section{Associated function spaces}\label{sec:space}

In this section, we  introduce global H\"older spaces to reflect the type of regularity we have obtained for the WTFBFs. These spaces are classically characterized using wavelets. This leads us to the second classical definition of these spaces through Littlewood-Paley analysis, which is extended naturally to Besov spaces.

\subsection{Weighted tensorized H\"older spaces for rectangular regularity}\label{sec:holder}

Let us introduce the weighted tensorized H\"older spaces as follows. 

\begin{definition}
    For  $s\in (0,1)$ and $\alpha \in [0,1]$, we define the  weighted tensorized H\"older space $T^{s, \alpha}C(\R^2)$ as the space of functions $f: \R^2 \to \R$ of $L^{\infty}(\R)$ such that there is $c>0$ with
    \begin{align*}
&|f(x_1+h_1,x_2+h_2)- f(x_1, x_2+h_2) - f(x_1+h_1,x_2)+ f(x_1,x_2)| \\
& \hspace{6cm} \leq c \, \min\{|h_1|,|h_2| \}^{(1+\alpha)s} \max\{|h_1|,|h_2| \}^{(1-\alpha)s},
    \end{align*}
    and 
$$|f(x_1+h_1,x_2)- f(x_1,x_2)| \leq c \, |h_1|^{(1+\alpha)s} \quad \text{and} \quad |f(x_1,x_2+h_2)-  f(x_1,x_2)| \leq c \,  |h_2|^{(1+\alpha)s}$$
for all $x_1,x_2,h_1,h_2 \in \R$. 
\end{definition}

In order to get a  characterization of this space in terms of wavelet coefficients, we first recall the definition of the hyperbolic wavelet bases as tensorial products of two unidimensional wavelet bases (see~\cite{devore:konyagin:temlyakov:1998}). 

\begin{definition}
Let $\psi$ denote a Lemari\'e-Meyer  wavelet  and $\varphi$ the associated scaling function. The hyperbolic wavelet basis is defined as the system $$\{\psi_{j_1,j_2,k_1,k_2} :  \, (j_1, j_2) \in (\N\cup \{-1\})^2, \, (k_1, k_2) \in \mathbb{Z}^2\}$$
where
\begin{itemize}
\item if $j_1,j_2\geq 0$,
$$
\psi_{j_1,j_2,k_1,k_2}(x_1,x_2)= \psi(2^{j_1}x_1-k_1)\psi(2^{j_2}x_2-k_2)\; ,
$$
\item if $j_1=-1$ and $j_2\geq 0$
$$
\psi_{-1,j_2,k_1,k_2}(x_1,x_2)= \varphi(x_1-k_1)\psi(2^{j_2}x_2-k_2)\; ,
$$
\item if $j_1\ge 0$ and $j_2=-1$
$$
\psi_{j_1,-1,k_1,k_2}(x_1,x_2)= \psi(2^{j_1}x_1-k_1)\varphi(x_2-k_2)\;,
$$
\item if $j_1=j_2=-1$
$$
\psi_{-1,-1,k_1,k_2}(x_1,x_2)= \varphi(x_1-k_1)\varphi(x_2-k_2)\;.
$$
\end{itemize}
For any $f\in\mathcal{S}'(\mathbb{R}^2)$, one then defines its hyperbolic wavelet coefficients by
\begin{eqnarray*}
c_{j_1,j_2,k_1,k_2}&=&2^{j_1+j_2}<f,\psi_{j_1,j_2,k_1,k_2}>\mbox{ if }j_1,j_2\geq 0\;,\\
c_{j_1,-1,k_1,k_2}&=&2^{j_1}<f,\psi_{j_1,j_2,k_1,k_2}>\mbox{ if }j_1\geq 0\mbox{ and }j_2=-1\;,\\
c_{-1,j_2,k_1,k_2}&=&2^{j_2}<f,\psi_{j_1,j_2,k_1,k_2}>\mbox{ if }j_1=-1\mbox{ and }j_2\geq 0\;,\\
c_{-1,-1,k_1,k_2}&=&<f,\psi_{j_1,j_2,k_1,k_2}>\mbox{ if }j_1=j_2=-1\;.
\end{eqnarray*}
\end{definition}

 The hyperbolic wavelet coefficients of a function $f$ provide a characterization of the space $T^{s, \alpha}C(\R^2)$, as shown in the following proposition. 

\begin{proposition}\label{prop:caractHolder}
Let $s\in (0,1)$, $\alpha \in [0,1]$ and consider $f: \R^2 \to \R$ such that $f \in \mathcal{S}'(\R)$.    Then $f \in T^{s, \alpha}C(\R^2)$ if and only if its hyperbolic wavelet coefficients satisfy
$$ 
\sup_{\bj\in (\N\cup\{-1\})^2,\bk\in \Z^2} 2^{( (1+\alpha)\max (\bj)+(1-\alpha) \min(\bj)) s}  | c_{\bj,\bk} | < + \infty.$$
\end{proposition}

\begin{proof}
    Since it is very classical, we just sketch the ideas of the proof. First, assume that $f \in T^{s, \alpha}C(\R^2)$. 
By using the vanishing moment of the wavelet, we can write
\begin{align*}
   c_{j_1,j_2,k_1,k_2}
    &  = 2^{j_1+j_2} \int_{\R^2} f(x_1,x_2) \psi(2^{j_1}x_1-k_1)\psi(2^{j_2}x_2-k_2) d\boldsymbol{x}\\[2ex]
    &  =  \int_{\R^2} f(\frac{y_1+k_1}{2^{j_1}},\frac{y_2+k_2}{2^{j_2}}) \psi(y_1)\psi(y_2) d\boldsymbol{y}\\[2ex]
    & =  \int_{\R^2} \Delta f_{\big(\frac{y_1}{2^{j_1}},\frac{y_2}{2^{j_2}}\big)}\big(\frac{k_1}{2^{j_1}}, \frac{k_2}{2^{j_2}}\big) \psi(y_1)\psi(y_2) d\boldsymbol{y}
\end{align*}
so that
\begin{align*}
|c_{j_1,j_2,k_1,k_2}|& \leq c\int_{\R^2}  \max\left\{\frac{|y_1|}{2^{j_1}},\frac{|y_2|}{2^{j_2}} \right\}^{(1-\alpha)s}\min\left\{\frac{|y_1|}{2^{j_1}},\frac{|y_2|}{2^{j_2}} \right\}^{(1+\alpha)s}\!|\psi(y_1)|\, |\psi(y_2) |d\boldsymbol{y}.
\end{align*}
Now, if $j_1\geq j_2$, we have
$$
\max\left\{\frac{|y_1|}{2^{j_1}},\frac{|y_2|}{2^{j_2}} \right\}^{(1-\alpha)s}\min\left\{\frac{|y_1|}{2^{j_1}},\frac{|y_2|}{2^{j_2}} \right\}^{(1+\alpha)s} \leq  2^{-(j_1 (1+\alpha) +j_2 (1-\alpha))s} \max\{ |y_1|,|y_2 |\}^{(1-\alpha)s} |y_1|^{(1+\alpha)s} 
$$
so that
$$
| c_{j_1,j_2,k_1,k_2}| \leq 2^{-(j_1 (1+\alpha) +j_2 (1-\alpha))s} \int_{\R^2 }\max\{ |y_1|,|y_2 |\}^{(1-\alpha)s} |y_1|^{(1+\alpha)s} \psi(y_1) \psi(y_2) dy_1dy_2.
$$
We proceed in the same way if $j_2 \geq j_1$. Now, if $j_1=-1$, we have
\begin{align*}
   c_{-1,j_2,k_1,k_2}
    &  = 2^{j_2} \int_{\R^2} f(x_1,x_2) \varphi(x_1-k_1)\psi(2^{j_2}x_2-k_2) d\boldsymbol{x}\\[2ex]
    &  =  \int_{\R^2} f(y_1+k_1,\frac{y_2+k_2}{2^{j_2}}) \varphi(y_1)\psi(y_2) d\boldsymbol{y}\\[2ex]
    & =  \int_{\R^2} \big(f(y_1+k_1,\frac{y_2+k_2}{2^{j_2}})  - f(y_1+k_1,\frac{k_2}{2^{j_2}}) \big)\varphi(y_1)\psi(y_2) d\boldsymbol{y}
\end{align*}
and we use the fact that 
$$
|f(y_1+k_1,\frac{y_2+k_2}{2^{j_2}})  - f(y_1+k_1,\frac{k_2}{2^{j_2}})| \leq c |\frac{y_2}{2^{j_2}}|^{(1+\alpha)s}
$$
to get the conclusion. The same argument holds for $j_2=-1$.

\smallskip

Let us now prove the converse result. We fix $x_1,x_2,h_1,h_2 \in \R$. Assume that $J_1,J_2 \in \N$ are such that 
$$
2^{-(J_1+1) }\leq|h_1|<2^{-J_1} \quad \text{et} \quad 2^{-(J_2+1) }\leq|h_2|<2^{-J_2}.  
$$
We have 
\begin{align}
  f& = \sum_{(\bj,\bk) \in \N^2 \times \Z^2} c_{\bj,\bk } \psi(2^{j_1}\cdot - k_1) \psi(2^{j_2}\cdot - k_2) + \sum_{j_2 \in \N,\bk \in \Z^2 } c_{-1,j_2,\bk}\varphi(\cdot - k_1) \psi(2^{j_2}\cdot - k_2) \nonumber\\
  & \quad + \sum_{j_1 \in \N,\bk \in \Z^2 }c_{j_1,-1,\bk} \psi(2^{j_1}\cdot - k_1) \varphi(\cdot - k_2) + \sum_{\bk \in \Z^2 }c_{-1,-1,\bk}\varphi(\cdot - k_1) \varphi(\cdot - k_2)  
\end{align}
so that the rectangular increment
$$f(x_1+h_1,x_2+h_2)- f(x_1, x_2+h_2) - f(x_1+h_1,x_2)+ f(x_1,x_2)
$$
can be decomposed into the four corresponding sums. For example, let us study the first contribution. The argument of the other sums can be obtained with an easy adaptation of the arguments. Note that the rectangular increment of $\psi(2^{j_1}\cdot - k_1) \psi(2^{j_2}\cdot - k_2)$ at $(x_1,x_2)$ of step $(h_1,h_2)$ is equal to 
$$
\big(\psi(2^{j_1}(x_1+h_1) - k_1)  -  \psi(2^{j_1}x_1 - k_1) \big) \big(\psi(2^{j_2}(x_2+h_2) - k_2) -\psi(2^{j_2}x_2 - k_2)\big)
$$
and it follows that one has to estimate
\begin{align}\label{eq:sum1}
 \sum_{(\bj,\bk) \in \N^2 \times \Z^2} |c_{\bj,\bk }| \,  \big|\psi(2^{j_1}(x_1+h_1) - k_1)  -  \psi(2^{j_1}x_1 - k_1) \big| \, \big|\psi(2^{j_2}(x_2+h_2) - k_2) -\psi(2^{j_2}x_2 - k_2)\big| .
\end{align}
We will use the two following upper bounds. First, for every $N \geq2$, there exists a constant $C_N>0 $ such that
\begin{align}\label{eq:sum00}
\sup_{x \in \R }\sum_{k \in \Z } \dfrac{1}{(1+|x-k|)^N} \leq C_N. 
\end{align}
Together with the fast decay of $\psi$, we obtain 
\begin{align}\label{eq:sum0}
\sup_{x \in \R} \sum_{k \in \Z} |\psi(2^jx-k)| \leq C_N
\end{align}
for all $j \in \N$. 
Seconldy, for every $j \in \N$, we set
   \begin{align*}
    F_{j} (x) = \sum_{k \in \Z} \psi(2^{j}x - k)  
    \end{align*}
    for all $x \in \R$. The regularity of $\psi$ and the mean value theorem imply that 
  \begin{align}\label{eq:sum2}
    |F_j(x+h) - F_j(x)| \leq 2^j|h| \sup_{\R }\sum_{k \in \Z} |D\psi(\cdot  - k)|  \leq C  2^j|h|
    \end{align}
    for some constant $C>0$ that does not depend on $j$, by using the fast decay of $\psi$ and Equation \eqref{eq:sum00}.  

Now,  assume that $J_1 <J_2$. The same argument will apply for $J_2 \leq J_1$ by symmetry. The quantity \eqref{eq:sum1} will be in turn decomposed into five sums: for $0 \leq j_1,j_2 < J_1$, for $0 \leq j_1 < J_1 \leq j_2 < J_2$, for $J_1 \leq j_1,j_2 < J_2$, for $0 \leq j_1<J_1<J_2 \leq j_2$ and finally for $j_1 \geq J_1$ and $j_2 \geq J_2$. First, using \eqref{eq:sum2}, we have
\begin{align*}
  & \sum_{0 \leq j_1,j_2 < J_1}  \sum_{\bk \in \Z^2} |c_{\bj,\bk }| \,  \big|\psi(2^{j_1}(x_1+h_1) - k_1)  -  \psi(2^{j_1}x_1 - k_1) \big| \, \big|\psi(2^{j_2}(x_2+h_2) - k_2) -\psi(2^{j_2}x_2 - k_2)\big| \\
  & \leq \sum_{j_1,j_2 \leq J_1} 2^{-\big( (1+\alpha)\max\{ j_1,j_2\}+(1-\alpha) \min\{j_1,j_2\}\big) s} C^2  2^{j_1+j_2}|h_1h_2| \\
  & \leq C^2 2^{-(J_1+J_2)} \left( \sum_{j_1=0}^{J_1-1} 2^{(1- (1+\alpha)s)j_1 }  \sum_{j_2=0}^{j_1}2^{(1-(1-\alpha)s)j_2 }   + \sum_{j_1=0}^{J_1-1} 2^{(1-(1-\alpha))sj_1}   \sum_{j_2=j_1+1}^{J_1-1} 2^{(1-(1+\alpha))sj_2} \right)\\
  & \leq C^2 2^{-(J_1+J_2)} \left( \sum_{j_1=0}^{J_1-1} 2^{(1- (1+\alpha)s)j_1 }  2^{(1-(1-\alpha)s)J_1 }   + \sum_{j_1=0}^{J_1-1} 2^{(1-(1-\alpha))sj_1}  2^{(1-(1+\alpha))sJ_2} \right)\\
   & \leq C^2 2^{-(J_1+J_2)} \left(  2^{(1- (1+\alpha)s)J_2 }  2^{(1-(1-\alpha)s)J_1 }   +  2^{(1-(1-\alpha))sJ_1}  2^{(1-(1+\alpha))sJ_2} \right)\\
   & \leq 2C^2 2^{-\big((1+\alpha)J_2 +(1-\alpha)J_1 \big)s }  
\end{align*}
since $J_1 < J_2$. We proceed similarly if $0 \leq j_1 < J_1 \leq j_2 < J_2$. Consider now the sum corresponding to the values of $\bj$ such that $J_1 \leq j_1,j_2 < J_2$. By using \eqref{eq:sum0} for the sum over $k_1$ and \eqref{eq:sum2} for the sum over $k_2$, one has 
\begin{align*}
  & \sum_{J_1 \leq j_1,j_2 < J_2}   \sum_{\bk \in \Z^2}|c_{\bj,\bk }| \,  \big|\psi(2^{j_1}(x_1+h_1) - k_1)  -  \psi(2^{j_1}x_1 - k_1) \big| \, \big|\psi(2^{j_2}(x_2+h_2) - k_2) -\psi(2^{j_2}x_2 - k_2)\big| \\
  & \leq  \sum_{J_1 \leq j_1,j_2 < J_2} \!\!c' 2^{-\big( (1+\alpha)\max\{ j_1,j_2\}+(1-\alpha) \min\{j_1,j_2\}\big) s}   C 2^{j_2} |h_2|\sum_{k_1} \left(\big|\psi(2^{j_1}(x_1+h_1) - k_1)\big| +\big|\psi(2^{j_1}x_1 - k_1)\big|  \right) \\
  & \leq 2^{-J_2}\sum_{j_1=J_1}^{J_2-1} \sum_{j_2 = J_1}^{j_1}  2^{-\big( (1+\alpha)j_1+(1-\alpha)j_2\big) s}  2C^2 \underbrace{2^{j_2}}_{\leq 2^{j_1}}  + 2^{-J_2}
  \sum_{j_1=J_1}^{J_2-1} \sum_{j_2 = j_1+1}^{J_2-1}   2^{-\big( (1+\alpha)j_2+(1-\alpha)j_1\big) s} 2C^2 2^{j_2}  \\
  & \leq 2 C^2 2^{-J_2} \sum_{j_1=J_1}^{J_2-1}2^{(1- (1+\alpha)s)j_1}  \sum_{j_2 = J_1}^{j_1}  2^{-(1-\alpha)sj_2}    + 2C^2 2^{-J_2}
  \sum_{j_1=J_1}^{J_2-1} 2^{-(1-\alpha)s j_1}  \sum_{j_2 = j_1+1}^{J_2-1}   2^{(1- (1+\alpha)s)j_2} \\
  & \leq 2 C^2 2^{-J_2} 2^{(1- (1+\alpha)s)J_2}   2^{-(1-\alpha)sJ_1} + 2C^2 2^{-J_2} 2^{-(1-\alpha)s J_1}   2^{(1- (1+\alpha)s)J_2} \\
  & \leq   4C^2  2^{ -\big((1+\alpha)J_2+(1-\alpha)J_1\big)s}.
  \end{align*}
  The sums corresponding to the values of $\bj$ with $0 \leq j_1<J_1<J_2 \leq j_2$ or with $j_1 \geq J_1$ and $j_2 \geq J_2$ are treated in a similar way. 
\end{proof}

\subsection{Hyperbolic Littlewood-Paley analysis and weighted tensorized H\"older spaces}

A classical approach for defining H\"older spaces without relying on finite differences involves the use of Littlewood-Paley analysis. This methodology also provides a straightforward framework for defining Besov spaces.  In this section, we adopt, as in reference \cite{ACJRV}, a hyperbolic Littlewood-Paley analysis to introduce new spaces that reflect weighted rectangular anisotropy. Notably, these spaces were independently and concurrently introduced in \cite{Har24} in the context of linear approximation problems.

The key feature of our approach is the weight that appears in the definition of these Besov spaces (see Definition \ref{defB} below) 
$$2^{((1+\alpha)\max(\bj) + (1-\alpha) \min(\bj))sq}. $$
This reveals connections with well-known function spaces: \begin{itemize}
    \item for $\alpha =0$, the weight simplifies to $$2^{((1+\alpha)\max(\bj) + (1-\alpha) \min(\bj))sq} = 2^{(j_1+j_2)sq}$$ leading to the Besov spaces with dominating mixed smoothness $S^s_{p,q}B(\R^2)$.
    \item for $\alpha=1$,  the weight becomes $$2^{((1+\alpha)\max(\bj) + (1-\alpha) \min(\bj))sq} = 2^{\max(\bj)2sq}$$ corresponding to the isotropic hyperbolic Besov spaces (with regularity $2s$).
\end{itemize}
However, when $\alpha \in (0,1)$, neither the hyperbolic Besov spaces nor the Besov spaces with dominating mixed smoothness are sufficient to characterize the type of rectangular anisotropy under consideration. This motivates the introduction of new spaces that bridge the gap between these existing frameworks.

Before introducing the hyperbolic Littlewood-Paley analysis, we first recall the classical Littlewood-Paley decomposition and the associated Besov spaces. After establishing these fundamental concepts, we turn to the hyperbolic Littlewood-Paley analysis. We revisit the hyperbolic Besov spaces and the Besov spaces with dominating mixed smoothness before defining the new spaces that reflect weighted rectangular anisotropy.
We then investigate the relationships between these new spaces and the previously mentioned Besov spaces and finally propose a wavelet characterization. This wavelet description aligns with the characterization provided for Hölder spaces in the previous subsection.

\subsubsection{Littlewood-Paley analysis and Besov spaces}
Let us recall the definition of classical Besov spaces, hyperbolic Besov spaces, and spaces with dominating mixed smoothness. 
The last two ones are defined with the hyperbolic Littlewood-Paley analysis but the first one is defined with a classical Littlewood-Paley analysis.

Let $\varphi_0  \ge 0$ belong to the Schwartz class ${\mathcal{S}}(\mathbb{R}^2)$ and be such that, for
$\boldsymbol{\xi}=(\xi_1,\xi_2) \in \mathbb{R}^2$,
$$
\varphi_0(\boldsymbol{\xi}) = 1 \quad \text{if} \quad \max\{|\xi_1|,|\xi_2| \}  \le 1\;,
$$
and
$$
\varphi_0(\boldsymbol{\xi})=0 \quad \text{if}  \quad \max\{ |2^{-1} \xi_1|, |2^{-1} \xi_2|\} \ge 1\;.
$$
For $j\in\mathbb{N}$, further define
\begin{align*}
\varphi_j(\boldsymbol{\xi}) &:= \varphi_0(2^{-j }\boldsymbol{\xi})-\varphi_0(2^{-(j-1)}\boldsymbol{\xi}) \\
&= \varphi_0(2^{-j }\xi_1,  2^{-j }\xi_2)-\varphi_0(2^{-(j-1)}\xi_1, 2^{-(j-1)}\xi_2)\,.
\end{align*}
Then
$\sum_{j\in \N} \varphi_j =  1$,
and  the sequence $(\theta_j)_{j \in \N}$ satisfies
\begin{equation*}
\mathrm{supp}\; (\varphi_0) \subset R_1, \quad \mathrm{supp}\; (\varphi_j)
\subset R_{j+1} \setminus R_{j-1}\;,
\end{equation*}
where
\begin{align}\label{rect}
R_j = \Big\lbrace \boldsymbol{\xi} \in \mathbb{R}^2\,:\; \max\{ |\xi_1|, |\xi_2|\}  \le 2^{j}
  \Big\rbrace \;.
\end{align}
For $f \in \mathcal{S}'(\mathbb{R}^2)$, we then define
$$
\Delta_j f := {\mathcal{F}}^{-1} ( \varphi_j {\mathcal{F}}f )\;.
$$
The sequence $(\Delta_j f )_{j \in \N}$ is called a {\it Littlewood-Paley analysis} of $f$. With this tool,
the Besov spaces are now defined as follows (see \cite{Tri06} for example for different equivalent definitions of spaces of smoothness, including anisotropic and weighted spaces). 

\begin{definition}
For $0<p ,q\le \infty$ and $s\in \mathbb{R}$, the Besov space $B^{s}_{p,q}(\mathbb{R}^2)$  is
defined by
$$
B^{s}_{p,q}(\mathbb{R}^2) = \Big\{ f \in \mathcal{S}'(\mathbb{R}^2)~:~ \, \Big( \sum_{j \in \mathbb{N}} 2^{jsq} \Vert
\Delta_j f \Vert_p^q \Big)^{1/q} <\infty \Big\}\,,
$$
with the usual modification for $q=\infty$.

This definition does not depend on chosen resolution of unity $\varphi_0$ and the quantity
$$
\Vert f \Vert_{B^{s}_{p,q}} = \Big( \sum_{j \in \mathbb{N}}  2^{jsq} \Vert \Delta^{}_jf \Vert_p^q
\Big)^{1/q}
$$
is a norm (resp. quasi-norm) on $B^{s}_{p,q}(\mathbb{R}^2)$ for $1 \leq p, \, q \leq \infty$ (resp. $0< \min
\{p,q\} <1$), with the usual modification if $q=\infty$.
\end{definition}

We also define the  Besov spaces with logarithmic scale. Again, we use the usual modification if $q= \infty$.
\begin{definition}
    For $0 \leq p, \, q \leq \infty$ and $s, \, \beta \in \R$, the Besov space with logarithmic scale is defined by 
\begin{equation}
 B^{s}_{p,q,\vert \log \vert^{\beta}}(\R^2) :  = \Big\{ f \in \mathcal{S}'(\mathbb{R}^2)~:~ \,  \sum_{j \in \N} j ^{-\beta q} 2^{jsq} \Vert \Delta_j f\Vert_p^q < \infty\big\}
\end{equation}
and we define a norm on $B^{s}_{p,q,\vert \log \vert^{\beta}}(\R^2)$ by
$$
\Vert f \Vert_{B^{s}_{p,q,\vert \log \vert^{\beta}}(\R^2)} = \left( \sum_{j \in \N} j ^{-\beta q} 2^{jsq} \Vert \Delta_j f \Vert_p^q \right)^{1/q}.
$$
\end{definition}

Subsequently, in order to compare the different spaces, we will use $\varphi_0$ as the tensor product of two $1$-dimensional functions. In other words, we consider $\varphi_0$ defined by
$$\varphi_0(\xi_1,\xi_2) = \theta_0(\xi_1) \theta_0(\xi_2)$$
where $\theta_0$ is a one-dimensional function.

\subsubsection{Hyperbolic Littlewood-Paley analysis}
Let $\theta_0 \in \mathcal{S}(\mathbb{R})$ be a non-negative function supported on $[-2,2]$ with $\theta_0 =1$ on $[-1,1]$. For any $j \in
\mathbb{N}_0$, let us further define 
$$
\theta_j = \theta_0(2^{-j} \cdot) - \theta_0(2^{-(j-1)} \cdot)
$$
such that $(\theta_{j})_{j \in \N}$ forms a univariate resolution of unity, i.e.,
$\sum_{j \in \N} \theta_j =1$. Observe that, for any $j \in \N_0$,
$$\mathrm{supp}(\theta_j) \subset \{ \boldsymbol{\xi} \in \R^2 : 2^{j-1} \le \vert \xi \vert \le 2^{j+1} \} \text{ and }\theta_j=\theta_1(2^{-(j-1)} \cdot).$$

\begin{remark}\label{rem:Meyer} 
In the following, the function $\theta_0$ can be chosen with an arbitrary compact support. It does not change the main
results even if technical details of proofs and lemmas have to be adapted. This allows to choose $\theta_0$ as the Fourier transform of a Meyer scaling function.
\end{remark}


\begin{definition}\label{def:hypLP} 
\begin{itemize}
    \item[(i)] For any $\bar{j} = (j_1,j_2) \in \N^2$ and any
$\boldsymbol{\xi}=(\xi_1,\xi_2) \in \mathbb{R}^2$ set
$$
\theta_{\bj}(\boldsymbol{\xi}) := \theta_{j_1}(\xi_1) \theta_{j_2}(\xi_2) \;.
$$
The function $\theta_{\bj}$ belongs to $\mathcal{S}(\mathbb{R}^2)$ for all $\bj \in \N_0^2$ and is compactly supported
on a dyadic rectangle. Further $\sum_{\bj \in \N^2}  \theta_{\bj} = 1$ and
$(\theta_{\bj})_{\bj \in \N^2}$ is called a hyperbolic resolution of unity.
\item[(ii)] For $f \in \mathcal{S}'(\mathbb{R}^2)$ and $\bj \in \N^2$ set
$$
\Delta_{\bj}f :=  {\mathcal{F}}^{-1} (\theta_{\bj} {\mathcal{F}}f )\;.
$$
The sequence $(\Delta_{\bj}f)_{\bj\in \N^2}$ is called a hyperbolic Littlewood-Paley analysis of $f$.
\end{itemize}
\end{definition}

Let us now introduce the hyperbolic Besov spaces with logarithmic scale, see \cite{ACJRV, SUV}.
\begin{definition}
   For $0 \leq p, \, q \leq \infty$ and $s, \, \beta \in \R$, the hyperbolic Besov space with hyperbolic scale is defined by
\begin{equation}
 \widetilde{B}^{s}_{p,q,\vert \log \vert^{\beta}}(\R^2)  :  = \Big\{ f \in \mathcal{S}'(\mathbb{R}^2)~:~ \,   \sum_{\bj \in \N^2 } (\max(\bj)) ^{-\beta q} 2^{\max(\bj)sq} \Vert \Delta_{\bj} f \Vert_p^q < \infty \Big\}
\end{equation}
with the norm
$$
\Vert f \Vert_{\widetilde{B}^{s}_{p,q,\vert \log \vert^{\beta}}(\R^2)} = \left( \sum_{\bj \in \N^2 } (\max(\bj)) ^{-\beta q} 2^{\max(\bj)sq} \Vert \Delta_{\bj} f \Vert_p^q \right)^{1/q},
$$
with the usual modification if $q= \infty$. 
\end{definition}
Note that the case $\beta=0$ gives the spaces without logarithmic correction. In \cite{ACJRV}, the following theorem was proved, which highlights the close connection between classical Besov spaces and their hyperbolic versions. This result serves as the foundation for providing a quasi-universal basis -- the hyperbolic one -- of numerous Besov and Triebel-Lizorkin spaces.

\begin{theorem}\label{Thm:BesovLog}
Let $s, \beta \in \R$ and $0 < p,q\leq \infty$ and $p'$ the conjugate exponent of $p$. We have the following embeddings
\begin{itemize}
\item if $q<+\infty$
\begin{equation}
\widetilde{B}^s_{p,q,\vert \log \vert^{\beta - r_1/q}}(\R^2) \hookrightarrow B^s_{p,q,\vert \log \vert^{\beta}}(\R^2) \hookrightarrow \widetilde{B}^s_{p,q,\vert \log \vert^{\beta +r_2/q}}(\R^2)
\end{equation}
where $$r_1 = \begin{cases}q(\frac{1}{p}-1)+ \max(q-1,0) & \text{ if } p \le 1 \\ 
\max( \frac{q}{\min(p,p')}-1,0) & \text{ if }p>1 \end{cases}$$
and
$$
r_2 = \begin{cases}1  & \text{ if } p < 1 \\ 
\max(1- \frac{q}{\max(p,p')}-1,0) & \text{ if }p\ge 1 .\end{cases}$$

\item if $q=+\infty$
\begin{equation}
\widetilde{B}^s_{p,\infty,\vert \log \vert^{\beta - \max(1/p-1,0) -1}}(\R^2) \hookrightarrow B^s_{p,\infty,\vert \log \vert^{\beta}}(\R^2) \hookrightarrow \widetilde{B}^s_{p,\infty,\vert \log \vert^{\beta }}(\R^2)
\end{equation}

\item if $q=2$, the Sobolev spaces $H^{s}_{2,\vert \log \vert^{\beta}}(\R^2)$ coincide with the classical Besov spaces $B^{s}_{2,2,\vert \log \vert^{\beta}}(\R^2)$ and with the hyperbolic Besov spaces $\widetilde{B}^{s}_{2,2,\vert \log \vert^{\beta}}(\R^2) $.
\end{itemize}
\end{theorem}
In \cite{SUV}, it has been shown that the hyperbolic and classical spaces coincide if and only if $p=q=2$. The logarithmic correction is then necessary (even maybe not optimal) to obtain the embeddings.

Finally, we introduce the spaces with dominating mixed smoothness, see \cite{Vybiral}.

\begin{definition}
    For $1 \leq p, \, q \leq \infty$ and $s \in \R$, the  Besov space with dominating mixed smoothness is defined by
\begin{equation}
 S^{s}_{p,q}B(\R^2)  :  = \Big\{ f \in \mathcal{S}'(\mathbb{R}^2)~:~ \,   \sum_{\bj \in \N^2 }  2^{(j_1+j_2)sq} \Vert \Delta_{\bj} f \Vert_p^q < \infty \Big\}
\end{equation}
with the norm
$$
\Vert f \Vert_{S^{s}_{p,q}B(\R^2)} = \left( \sum_{\bj \in \N^2 }  2^{(j_1+j_2)sq} \Vert \Delta_{\bj} f \Vert_p^q \right)^{1/q},
$$
with the usual modification if $q= \infty$. 
\end{definition}

\subsubsection{Definition of the weighted tensorized Besov spaces}

As mentioned before, to address the cases where $\alpha \in (0,1)$, we introduce a new class of Besov spaces characterized through their hyperbolic Littlewood-Paley analysis. This approach provides a refined framework that bridges the gap between the hyperbolic Besov spaces and spaces with dominating mixed smoothness.

\begin{definition}\label{defB} For $0<p ,q \le \infty$,  $s\in \mathbb{R}$ and $0\le \alpha \le 1$,  the weighted tensorized Besov space $T^{s,\alpha}_{p,q}{B}(\mathbb{R}^2)$ is defined by
$$
T^{s,\alpha}_{p,q}B(\mathbb{R}^2) = \Big\{ f \in \mathcal{S}'(\mathbb{R}^2)~:~\Big( \sum_{\bj \in \N^2}
2^{((1+\alpha)\max(\bj) + (1-\alpha) \min(\bj))sq} \Vert \Delta_{\bj}f \Vert_p^q \Big)^{1/q} <  \infty
\Big\}\,.
$$
The quantity 
$$
\|f\|_{T^{s,\alpha}_{p,q}B(\mathbb{\R}^2 )}:= \Big( \sum_{\bj \in \N^2}  2^{((1+\alpha)\max(\bj) + (1-\alpha) \min(\bj))sq}
\Vert \Delta_{\bj}f \Vert_p^q \Big)^{1/q}
$$
is a norm (resp. quasi-norm) on $T^{s,\alpha}_{p,q}B(\mathbb{R}^2)$ for $1 \leq p, \, q \leq \infty$
(resp. $0< \min\{p,q\} <1$). 
We adopt the usual modification if $q=\infty$.
\end{definition}

The definition is independent of the chosen hyperbolic partition of unity. This follows from the proof of Proposition 1 (p. 87) in \cite{ST}, which addresses the case of spaces with dominating mixed smoothness.

We will now turn to the embeddings between the tensorized Besov spaces and the other spaces.  We have the following results.

\begin{proposition}\label{Prop:emb}
For  $0 <p,q \leq \infty$, $s \in \R$  and $0 \le \alpha \le 1$, one has
\begin{itemize}
\item $S^{(1+\alpha)s}_{p,q}B(\R^2) \hookrightarrow T^{s,\alpha}_{p,q}B(\R^2) \hookrightarrow S^{s}_{p,q}B(\R^2)$,
\item $\widetilde{B}^{2s}_{p,q}(\R^2) \hookrightarrow T^{s,\alpha}_{p,q}B(\R^2) \hookrightarrow \widetilde{B}^{(1+\alpha)s}_{p,q}(\R^2)$,
\end{itemize}
and these embeddings are optimal. 
\end{proposition}
\begin{proof}
The two first items are obvious by definition of the different Littlewood-Paley analysis since
$$
 (j_1+j_2) = \min(j_1,j_2)+ \max (j_1,j_2) \le (1-\alpha) \min (j_1,j_2) + (1+ \alpha) \max(j_1,j_2) \le (1+\alpha)(j_1+j_2)
$$
and
$$
(1+\alpha) \max(j_1,j_2) \le (1-\alpha) \min(j_1,j_2) + (1+\alpha) \max(j_1,j_2) \le 2 \max(j_1,j_2).
$$

Le us prove that these embeddings are optimal. Let $\varepsilon >0$. We consider a one-dimensional function $g$ such that $g \in B^{(1+ \alpha)s}_{p,q}(\R)$ but $g \notin B^{(1+ \alpha)s+\varepsilon}_{p,q}(\R)$. Additionally, we take another one-dimensional function $u\neq 0$ such that $\mbox{Supp} \,  (\widehat{u}) \subset [-1,1]$. We define $f$ as $f(x_1,x_2) = u(x_1)g(x_2)$. Using the localization of the supports $\widehat{u}$ and of $\theta_j$, one has, for $(j_1,j_2) \in \N_0^2$
\begin{equation*}
\Vert \Delta_{j_1,j_2}f \Vert_p= \begin{cases} \Vert \Delta_0(u) \Vert_p \Vert \Delta_{j_2} g \Vert_p & \text{ if $j_1=0$}, \\
 0 &\text{ if $j_1\ge 1$}.
\end{cases}
\end{equation*}
It follows that 
\begin{align*}
\Vert f \Vert_{S^{(1+\alpha)s}_{p,q}B(\R^2)}= \Vert f \Vert_{T^{s,\alpha}_{p,q}B(\R^2)} = \Vert f \Vert_{\widetilde{B}^{(1+\alpha)s}_{p,q}(\R^2)}& =\left( \Vert \Delta_0 u \Vert_p^q \sum_{j_2 \in \N}  2^{j_2(1+\alpha)sq} \Vert \Delta_{j_2}  g \Vert_{p}^q \right)^{1/q}\\
&= \Vert \Delta_0 u \Vert_p \Vert  g \Vert_{B^{(1+\alpha)s}_{p,q}(\R^2)}
\end{align*}
which proves the optimality of the left embedding of the first item and of the right embedding of the second item.

For the two other ones, we consider the function $f$ such that $$\widehat{f}(\xi_1,\xi_2) = \sum_{j \ge 1}\frac{1}{j^{2/q}}2^{-2js}2^{-2(1-1/p)j} \widehat{v} (2^{-j} \xi_1)\widehat{v}(2^{-j} \xi_2)$$ 
 where $v$ is a non-null function defined on $1 \le \vert \xi \vert \le 2$ such that $\Delta_1 (v)= v$ and $\Delta_j (v)=0$ for $j \neq 1$. It follows that
\begin{eqnarray*}
\Vert f \Vert_{S^{s}_{p,\infty}B(\R^2)} = \Vert f \Vert_{T^{s,\alpha}_{p,\infty}B(\R^2)} = \Vert f \Vert_{{\widetilde{B}}^{2s}_{p,\infty}(\R^2)} & = & \left( \sum_{j \ge 1} \frac{1}{j^2} 2^{2(1-1/p)jq}\Vert  2^{2j} v(2^j x_1) v(2^jx_2) \Vert_p^q \right)^{1/q}\\
& = & \Vert v \Vert_p^2 \left( \frac{\pi^2}{6} \right)^{1/q}
\end{eqnarray*}
which implies the other embeddings. 
\end{proof}

Finally, the combination of Theorem \ref{Thm:BesovLog} and this proposition provides immediate embeddings of the weighted tensorized spaces in the classical Besov space with a logarithmic correction. It can be improve in the case $p=q=2$ since the logarithmic correction is no more necessary and for some particular values of $q$ (see Theorem 5.5 of \cite{Remo24}).

\subsubsection{Wavelet characterization of the weighted tensorized Besov spaces}
We define a space of sequences by the following condition
\begin{equation}\label{eq:Tspq} 
t^{s,\alpha}_{p,q}b := \left\{ (c_{\bj,\bk}) : \, \sum_{\bj \in {(\N \cup \{-1\})^2}} 2^{-\frac{(j_1+j_2)q}{p}} 2^{-((1+\alpha)\max(\bj)+(1-\alpha) \min(\bj) )sq} \left( \sum_{\bk \in \Z^2 } \vert c_{\bj,\bk} \vert^p \right)^{q/p} <+\infty \right\}
\end{equation}
and we define a (pseudo-)norm on $t^{s,\alpha}_{p,q}b$ for $c=(c_{\bj,\bk})$ by
$$
\Vert c \Vert_{t^{s,\alpha}_{p,q}b} := \left( \sum_{\bj \in {(\N \cup \{-1\})^2}} 2^{-\frac{(j_1+j_2)q}{p}} 2^{-((1+\alpha)\max(\bj)+(1-\alpha) \min(\bj) )sq} \left( \sum_{\bk \in \Z^2 } \vert c_{\bj,\bk} \vert^p \right)^{q/p} \right)^{1/q}.
$$
with the usual modification if $q=+\infty$:
$$
\Vert c \Vert_{t^{s,\alpha}_{p,\infty}b} := \max_{\bj \in {(\N \cup \{-1\})^2}} 2^{-\frac{(j_1+j_2)}{p}} 2^{-((1+\alpha)\max(\bj)+(1-\alpha) \min(\bj) )s} \left( \sum_{\bk \in \Z^2 } \vert c_{\bj,\bk} \vert^p \right)^{1/p}.
$$
We prove the following characterization in hyperbolic wavelets.
\begin{theorem}\label{thm:caract}
Let $0 <p,q \leq \infty$, $s \in \R$ and $0 \le \alpha \le 1$ and let $ (c_{\bj,\bk})$ denote the sequence of the Meyer wavelet coefficients of a function $f \in \mathcal{S}'(\R^2)$. The following assertions are equivalent:
\begin{enumerate}
\item $f \in T^{s,\alpha}_{p,q}B(\R^2)$,
\item $ (c_{\bj,\bk} )\in t^{s,\alpha}_{p,q}b $.
\end{enumerate}
Moreover, there exist two constants $C_1,C_2>0$ such that
$$
C_1 \Vert  (c_{\bj,\bk}) \Vert_{t^{s,\alpha}_{p,q}b} \le \Vert f \Vert_{T^{s,\alpha}_{p,q}B(\R^2)} \le C_2 \Vert (c_{\bj,\bk}) \Vert_{t^{s,\alpha}_{p,q}b} \, . 
$$
\end{theorem}

\begin{remark}
Note that the statement is given with a $L^{\infty}$ normalisation for the wavelet functions. 
\end{remark}

The proof of Theorem \ref{thm:caract} is an adaptation of the results of \cite{FJ85} and their extension to the tensor product case in \cite{ACJRV}.
In particular, we will rely on the following lemma, established in \cite{ACJRV} (Lemma 4.5),  which adapts Lemma 2.4 of \cite{FJ85} to  the case of rectangular supports.

\begin{lemma}\label{lem:FJ}
Let $0<p \leq \infty$ and $\bj=(j_1,j_2) \in \N^2$. Assume that $g \in \mathcal{S}'(\R^2)$ and  that $\mbox{Supp}\, (\widehat{g}) \subset \{ \boldsymbol{\xi} \in \R^2: \, \vert \xi_1 \vert \le 2^{j_1+1} \text{ and } \vert \xi_2 \vert \le 2^{j_2+1} \}$. Then, there exists $C>0$ such that
$$
\left( \sum_{\bk \in \Z^2 } 2^{-(j_1+j_2)} \left\vert g\left( k_12^{-j_1}, k_2 2^{-j_2} \right) \right\vert^p \right)^{1/p} \le C \, \Vert g \Vert_{L^p}.
$$
\end{lemma}

We will also use the following adaptation of Lemma 3.4 of \cite{FJ85}, given in Lemma 4.6 \cite{ACJRV} which is useful do deal with the case $p>1$.

\begin{lemma} \label{lem:FJ2}
Let $1 \le p \le \infty$  and $\ell_1, \ell_2, m_1, m_2$ be integers such that $\ell_1 \le m_1$ and $\ell_2 \le m_2$. Assume that $g_{\bk}$, $\bk \in \Z^2$, are functions satisfying the following inequality 
\begin{equation}\label{eq:ineqg}
\forall {\boldsymbol{x}} \in \R^2, \, \vert g_{\bk}(\boldsymbol{x}) \vert \le \frac{C}{\left( 1+2^{\min(\ell_1,m_1)}|x_1 - 2^{-m_1}k_1 \vert \right)^2 \left( 1+2^{\min(\ell_2,m_2)}|x_2 - 2^{-m_2}k_2 \vert \right)^2}.
\end{equation}
for some $C>0$. If one sets  
$$
F = \sum_{\bk \in \Z^2} d_{\bk} g_{\bk},
$$
then
\begin{equation}\label{eq:Fp}
\Vert F \Vert_{L^p} \le C \, 2^{-(m_1+m_2)/p} 2^{m_1-\ell_1}2^{m_2 - \ell_2} \left( \sum_{\bk \in \Z^2} \vert d_{\bk} \vert^p \right)^{1/p}.
\end{equation}
\end{lemma}
\begin{proof}[Proof of Theorem \ref{thm:caract}]
Since the Littlewood-Paley analysis does not depend of the function $\theta$, we can chose $\theta=\hat{\varphi}$ where $\varphi$ is the Meyer scaling function, see Remark \ref{rem:Meyer} above.

For the first implication, let us observe that $c_{\bj+1,\bk} = \Delta_{\bj} f(2^{-j_1}k_1,2^{-j_2}k_2)$ where we use the notation $\bj+ 1 = (j_1+1,j_2+1)$. Applying Lemma~\ref{lem:FJ} to the function $g = \Delta_{\bj}f \in \mathcal{S}(\R^2)$, we obtain
$$
\sum_{\bk \in \Z^2} \vert c_{\bj+1,\bk} \vert^p = \sum_{\bk \in \Z^2} \vert \Delta_{\bj}f(2^{-j_1}k_1,2^{-j_2}k_2) \vert^p \le C 2^{j_1+j_2} \Vert \Delta_{\bj} f \Vert_p^p,
$$
which gives $ (c_{\bj,\bk} )\in t^{s,\alpha}_{p,q}b $ and the upper-bound
$$
\Vert c_{\bj,\bk} \Vert_{t^{s,\alpha}_{p,q}b} \le \Vert f \Vert_{T^{s,\alpha}_{p,q}B(\R^2)}.
$$

We focus now on the converse implication. We mainly adapt the proof of \cite{ACJRV} and first deal with the case $0<p<1$. We have to bound $\Vert \Delta_{\bj}(f) \Vert_p = \Vert \theta_{\bj} \ast f \Vert_p$. We have
$$
({\mathcal{F}}^{-1}\theta_{\bj}) \ast f = \sum_{\bar{m} \in (\N\cup\{-1\})^2} \sum_{\bar{k}\in \Z^2} c_{\bar{m},\bar{k}}(\theta_{\bj} \ast \psi_{\bar{m},\bk}).
$$

Since $\theta_{\bj}$ and $\psi_{\bj,\bk}$ are both tensor products of one-dimensional functions, Lemma 3.3. of \cite{FJ85} can be applied and give the existence of $C>0$ such that for any $r>0$ such that for all ${\boldsymbol{x}} \in \R^2$, one has
\begin{equation}\label{eq:convol}
\vert ({\mathcal{F}}^{-1}\theta_{\bar{j}} )\ast \psi_{\bar{m},\bk} ({\boldsymbol{x}}) \vert \le C \frac{2^{-\Vert \bj - \bar{m} \Vert_1(M+3) }}{\left( 1+2^{\min(j_1,m_1)}\vert x_1 - 2^{-m_1}k_1 \vert \right)^r \left( 1+2^{\min(j_2,m_2)}\vert x_2 - 2^{-m_2}k_2 \vert \right)^r},
\end{equation}
where $M$ is taken smaller or equal to the number of vanishing moments of the wavelets.

Because of the localization of the Fourier transform of the Meyer wavelet and the shift of index - starting at $-1$ for the wavelets, we notice that
$$
\Delta_{\bj}f =  \sum_{\bk \in \Z^2}\sum_{m_1=j_1-2}^{j_1} \sum_{m_2=j_2-2}^{j_2} c_{m_1,m_2,k_1,k_2} ({\mathcal{F}}^{-1}\theta_{j_1,j_2})\ast \psi_{m_1,m_2,k_1,k_2}.
$$
By the concavity of $x \to x^p$ for $0<p<1$, it follows that
$$
\vert \Delta_{\bj}f(x)\vert^p \le \sum_{m_1=j_1-2}^{j_1}\sum_{m_2 = j_2-2}^{j_2} \sum_{\bk \in \Z^2} \vert c_{m_1,m_2,k_1,k_2}\vert^p \vert (({\mathcal{F}}^{-1}\theta_{j_1,j_2})\ast \psi_{m_1,m_2,k_1,k_2})(x)\vert^p.
$$
Hence, with inequality \eqref{eq:convol}, we obtain
$$
\vert \Delta_{\bj}f(x) \vert^p \le 
\sum_{m_1=j_1-2}^{j_1}\sum_{m_2 = j_2-2}^{j_2} \sum_{\bk \in \Z^2} \vert c_{m_1,m_2,k_1,k_2}\vert^p \frac{C}{(1+2^{j_1}\vert x_1-2^{-j_1} k_1\vert)^{rp}(1+2^{j_2} \vert x_2-2^{-m_2}k_2\vert)^{rp} }
$$
An integration over $\R^2$ with a change of variable $u_i = x_i - 2^{j_i}k_i$ for $i=1,2$ gives
$$
\Vert \Delta_{\bj}f \Vert_p^p \le C 2^{-(j_1+j_2)}\sum_{m_1=j_1-2}^{j_1}\sum_{m_2 = j_2-2}^{j_2} \sum_{\bk \in \Z^2}\vert c_{m_1,m_2,k_1,k_2}\vert^p 
$$
and 
\begin{eqnarray*}
\Vert f \Vert_{T^{s,\alpha}_{p,q}} & = & \left( \sum_{\bj \in \N^2} 2^{((1+\alpha)\max(\bj)+(1-\alpha)\min(\bj))sq}\Vert \Delta_{\bj}f \Vert_p^q \right)^{1/q} \\
& \le & C\left( \sum_{\bj \in \N^2} 2^{((1+\alpha)\max(\bj)+(1-\alpha)\min(\bj))sq} 2^{-(j_1+j_2)/p}
\sum_{m_1=j_1-2}^{j_1} \sum_{m_2=j_2-2}^{j_2}\sum_{\bk \in \Z^2} \vert c_{m_1,m_2,k_1,k_2} \vert^p  \right)^{1/q} \\
& \le & \widetilde{C} \left( \sum_{\bj \in \N^2} 2^{((1+\alpha)\max(\bj)+(1-\alpha)\min(\bj))sq} 2^{-(j_1+j_2)/p}
 \sum_{\bk \in \Z^2} \vert c_{j_1,j_2,k_1,k_2} \vert^p \right)^{1/q} \\
 & \le & \widetilde{C} \Vert c_{\bj,\bk} \Vert_{t^{s,\alpha}_{p,q}b}.
\end{eqnarray*}
We now consider the case $p\geq 1$. One can write

$$
\Delta_{\bj}f =\sum_{m_1=j_1-2}^{j_1} \sum_{m_2=j_2-2}^{j_2} \sum_{\bk \in \Z^2} c_{m_1,m_2,k_1,k_2} g_{m_1,m_2,k_1,k_2}
$$
with
$$
g_{m_1,m_2,k_1,k_2} =  {\mathcal{F}}^{-1}(\theta_{j_1,j_2}) \ast \psi_{m_1,m_2,k_1,k_2}.
$$ 
Again, this is due to the fact that $\theta_{\bar{j}}$ and $\widehat{\psi_{\bj+1, \bk}}$ have the same support which meet at most three dyadic annuli.
Lemma \ref{lem:FJ2} gives that
\begin{eqnarray*}
\Vert \Delta_{\bj} f \Vert_{L^p } & \le  & \sum_{m_1=j_1-2}^{j_1} \sum_{m_2=j_2-2}^{j_2} \Vert  \sum_{\bk \in \Z^2} c_{m_1,m_2,k_1,k_2} g_{m_1,m_2,k_1,k_2} \Vert_p \\
&  \le & C \sum_{m_1=j_1-2}^{j_1} \sum_{m_2=j_2-2}^{j_2}  2^{- (m_1+m_2)/p}\left( \sum_{k_1,k_2} \vert c_{m_1,m_2,k_1,k_2} \vert^p \right)^{1/p}.
 \end{eqnarray*}
Finally, we obtain that
\begin{eqnarray*}
&&\Vert f \Vert_{T^{s,\alpha}_{p,q}B(\R^2)}^q \\[1ex]
& = & \sum_{\bj \in (\N\cup\{-1\})^2} 2^{((1+\alpha)\max(\bj) + (1-\alpha)\min(\bj)) s q} \Vert \Delta_{\bj}(f) \Vert_{L^p}^q \\
& \le & C \sum_{\bj\in (\N\cup\{-1\})^2} \sum_{m_1=j_1-2}^{j_1} \sum_{m_2=j_2-2}^{j_2}2^{((1+\alpha)\max(\bj) + (1-\alpha)\min(\bj)) s q} 2^{-q(m_1+m_2)/p} \left( \sum_{\bk} \vert c_{\overline{m},\bk} \vert^p \right)^{q/p} \\
& \le & \widetilde{C} \sum_{\bj\in (\N\cup\{-1\})^2}2^{((1+\alpha)\max(\bj) + (1-\alpha)\min(\bj)) s q} 2^{-q(j_1+j_2)/p} \left( \sum_{k_1,k_2} \vert c_{\bj,\bk} \vert^p \right)^{q/p} 
\end{eqnarray*}
for some constant $\widetilde{C} >0$, which is the desired inequality. 
\end{proof}

Let us observe that, in the particular case where $ q = p = \infty$, Theorem \ref{thm:caract} provides a simplified characterization of the spaces $  T^{s, \alpha}_{\infty,\infty}B(\mathbb{R}^2)$. Specifically, a function $f$ belongs to $T^{s, \alpha}_{\infty,\infty}B(\mathbb{R}^2)$  if and only if 
$$
\sup_{\bj \in (\N \cup \{ -1\})^2, \bk \in \Z^2}2^{- ((1+ \alpha)\max(\bj) + (1-\alpha) \min (\bj))s} |c_{\bj,\bk} |  < + \infty.
$$ 
where $(c_{\bj,\bk})$ denotes the wavelet coefficients of $f$ in the Meyer basis. 
This result, combined with Proposition \ref{prop:caractHolder}, allows us to conclude that the spaces  $T^{s, \alpha}_{\infty,\infty}B(\mathbb{R}^2)$ correspond precisely to the weighted tensorized H\"older spaces $T^{s, \alpha}C(\R^2)$ introduced in Subsection \ref{sec:holder}.

\begin{theorem}
The Meyer wavelet basis is an unconditional basis of $T^{s,\alpha}_{p,q}B(\R^2)$.
\end{theorem}
\begin{proof}
Consider $f \in T^{s,\alpha}_{p,q}B(\R^2)$. Then its wavelet coefficients satisfy 
$$
\Vert c_{\bj,\bk} \Vert_{t^{s,\alpha}_{p,q}b} <+\infty.
$$
Since $f -\sum_{-1 \le j_1,j_2 \le J}\sum_{0 \le k_1,k_2 \le K} c_{\bj,\bk}\psi_{\bj,\bk}$ has for wavelet coefficients 
$$
d_{\bj,\bk}= \begin{cases} 0 & \text{ if } -1 \le j_1,j_2 \le J \text{ and } 0 \le k_1,k_2 \le K   \\ c_{\bj,\bk} & \text{ otherwise, } \end{cases}$$  it implies that 
$$
\big\Vert f - \sum_{-1 \le j_1,j_2 \le J}\sum_{0 \le k_1,k_2 \le K} c_{\bj,\bk}\psi_{\bj,\bk} \, \big\Vert_{T^{s,\alpha}_{p,q}B(\R^2)} \le \Vert d_{\bj,\bk} \Vert_{t^{s,\alpha}_{p,q}b}
$$
is the tail of a convergent series, so that the functions $\psi_{\bj,\bk}$, $\bj \in \N^2, \bk \in \Z^2$,  is a Schauder basis of $T^{s,\alpha}_{p,q}B(\R^2)$. Finally if $\varepsilon_{\bj,\bk}= \pm 1$, the signed series converges in $S^{s}_{p,q}B(\R^2)$ (for which the Meyer wavelet basis is unconditional, see \cite{SUV}) to a function $g$. By the wavelet characterization, we obtain that $g$ belongs to $T^{s,\alpha}_{p,q}B(\R^2)$.
\end{proof}
\newpage
\bibliographystyle{plain}
\bibliography{RegularityWTFBF.bib}

\end{document}